\newcommand{\DG}[2]{%
  \ifdefined\FORDG
    #1
  \else
    #2
  \fi
}
\DG{
\documentclass{degruyter-proceedings}
}{
\documentclass{article}
\pdfoutput=1
}

\usepackage[colorlinks=true,linkcolor=blue,citecolor=blue]{hyperref}

\usepackage{graphicx}
\usepackage{color}
\usepackage{tikz}

\usepackage[nocompress]{cite}

\usepackage[latin1]{inputenc}
\usepackage[T1]{fontenc}
\usepackage{microtype}

\usepackage{amsmath,amsfonts,amsthm,bm}
\allowdisplaybreaks
\DG{}{\usepackage[a4paper]{geometry}}

\newcommand{\Z}{\mathbb{Z}} % integers
\newcommand{\R}{\mathbb{R}} % reals
\newcommand{\N}{\mathbb{N}} % natural numbers {1, 2, ...}
\newcommand{\F}{\mathbb{F}} % field, finite field

\newcommand{\imagunit}{\mathrm{i}}
\newcommand{\twopii}{2 \pi \imagunit \;}

\newcommand{\setu}{\mathfrak{u}}
\newcommand{\setv}{\mathfrak{v}}

\DeclareMathOperator{\wal}{wal}

\newcommand{\floor}[1]{\lfloor #1 \rfloor}

\newcommand{\rd}{\,\mathrm{d}} % differential symbol for use in integrals

\newcommand{\bszero}{\boldsymbol{0}} % vector of zeros
\newcommand{\bsh}{\boldsymbol{h}}    % vector h
\newcommand{\bsx}{\boldsymbol{x}}    % vector x
\newcommand{\bsy}{\boldsymbol{y}}    % vector y
\newcommand{\bsz}{\boldsymbol{z}}    % vector z
\newcommand{\bsDelta}{\boldsymbol{\Delta}}    % vector \Delta
\newcommand{\bsgamma}{\boldsymbol{\gamma}}    % vector \gamma
\newcommand{\bssigma}{\boldsymbol{\sigma}}    % vector \sigma

\newcommand{\E}{\mathrm{e}} % exponential
\newcommand{\tpmod}[1]{~(\operatorname{mod}{#1})} % a tiny pmod for under sum symbols

\newcommand{\X}{\textsf{x}}

\DeclareSymbolFont{bbold}{U}{bbold}{m}{n}
\DeclareSymbolFontAlphabet{\mathbbold}{bbold}
\newcommand{\ind}[1]{\mathbbold{1}_{#1}}

\def\citep#1#2{\cite[{#1}]{#2}}

\DG{
}{
\theoremstyle{plain}
  \newtheorem{theorem}{Theorem}
  
  \newtheorem{lemma}{Lemma}
  \newtheorem{corollary}{Corollary}
\theoremstyle{definition}

\theoremstyle{remark}
  
}

\newcommand{\RefEq}[1]{~\textup{(\ref{#1})}}
\newcommand{\RefEqTwo}[2]{~\textup{(\ref{#1})} and~\textup{(\ref{#2})}}

\newcommand{\RefSec}[1]{Section~\textup{\ref{#1}}}
\newcommand{\RefSecTwo}[2]{Sections~\textup{\ref{#1}} and~\textup{\ref{#2}}}

\newcommand{\RefThm}[1]{Theorem~\textup{\ref{#1}}}

\newcommand{\RefFig}[1]{Figure~\textup{\ref{#1}}}

\def\myabstract{
A comprehensive overview of lattice rules and polynomial lattice rules is given for function spaces based on $\ell_p$ semi-norms.
Good lattice rules and polynomial lattice rules are defined as those obtaining worst-case errors bounded by the optimal rate of convergence for the corresponding function space.
The focus is on algebraic rates of convergence $O(N^{-\alpha+\epsilon})$ for $\alpha \ge 1$ and any $\epsilon > 0$, where $\alpha$ is the decay of a series representation of the integrand function.
The dependence of the implied constant on the dimension can be controlled by weights which determine the influence of the different coordinates.
Different types of weights are discussed.
The construction of good lattice rules, and polynomial lattice rules, can be done using the same method for all $1 < p \le \infty$; but the case $p=1$ is special from the construction point of view.
For $1 < p \le \infty$ the component-by-component construction and its fast algorithm for different weighted function spaces is then discussed.
}
\def\myack{
The author wants to thank the anonymous referee, the editors and Gowri Suryanarayana for very carefully reading through this manuscript and making useful suggestions.
}

\DG{
\title{The construction of good lattice rules \\ and polynomial lattice rules}
\headlinetitle{The construction of good lattice rules and polynomial lattice rules}

\lastnameone{Nuyens}
\firstnameone{Dirk}
\nameshortone{D.~Nuyens}
\addressone{Department of Computer Science, KU Leuven, Celestijnenlaan 200A, 3000 Leuven}
\countryone{Belgium}
\emailone{dirk.nuyens@cs.kuleuven.be}
\DG{
\abstract{
\myabstract
}


\acknowledgments{
\myack
}
}{}
}{
\title{The construction of good lattice rules \\ and polynomial lattice rules}
\author{Dirk Nuyens \\ \normalsize Department of Computer Science \\ \normalsize KU~Leuven, Belgium}
\date{}
}

\begin{document}

\maketitle

\DG{}{
\begin{abstract}
\myabstract
\end{abstract}
}

\DG{}
{{
\renewcommand*{\thefootnote}{\fnsymbol{footnote}}
\footnotetext{\today}
}}

\section{Lattice rules and polynomial lattice rules}

The aim is to approximate multivariate integrals over the $s$-dimensional unit cube 
\begin{align*}
  I(f)
  &:=
  \int_{[0,1)^s} f(\bsx) \rd\bsx
\end{align*}
by equal-weight cubature rules of the form
\begin{align}\label{eq:QN}
  Q_N(f; \{\bsx_k\}_{k=0}^{N-1})
  &:=
  \frac1N \sum_{k=0}^{N-1} f(\bsx_k)
  .
\end{align}
It is well known that for all Riemann integrable functions $f$, $Q_N(f) \to I(f)$ for $N \to \infty$ if and only if the sequence $\{\bsx_k\}_{k=0}^\infty$ is uniformly distributed, see, e.g., \cite{KN74,Nie92}.
The measure of non-uniformity is called \emph{discrepancy}.
The $L_\infty$-star-discrepancy measures the discrepancy between the true uniform distribution and the point set by comparing boxes of the form $[\bszero, \bsx) = \prod_{j=1}^s [0,x_j)$ for $\bsx = (x_1,\ldots,x_s)$:
\begin{align*}
  D_N^*(\bsx_0, \ldots, \bsx_{N-1})
  &:=
  \sup_{\bsx \in [0,1]^s} \left| \frac{|\{\bsx_k\}_{k=0}^{N-1} \cap [\bszero,\bsx)|}{N} - \operatorname{vol}([\bszero,\bsx)) \right|
  .
\end{align*}
Point sets and sequences which have a discrepancy of $O(N^{-1} (\log N)^{s})$, or even $O(N^{-1} (\log N)^{s-1})$ in case of a finite point set, are called low-discrepancy point sets and sequences, see \cite{Nie92} for a general reference.

Two related families of such point sets are studied in this manuscript: lattice rules and polynomial lattice rules.
Some classical references on lattice rules are \cite{Kor63,SJ94,Nie92}.
Polynomial lattice rules follow a very similar construction procedure but are in fact a kind of digital net, see \cite{Nie92,DP2010,Pil2012} for a reference on those.
In both cases only rank-$1$ (polynomial) lattices will be considered, these are lattices generated by a single generating vector modulo the modulus of the point set (which is the number of points for lattice rules and a polynomial in the case of polynomial lattice rules).
In both cases, the generating vector will determine the quality of the lattice when used for numerical integration.
\RefFig{fig:lattices} depicts the node set of a lattice rule and of a polynomial lattice rule.

\begin{figure}
  \centering
  \includegraphics{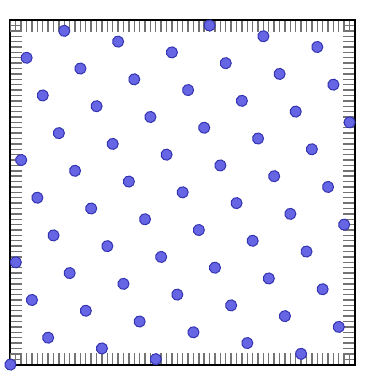}
  \qquad \qquad
  \includegraphics{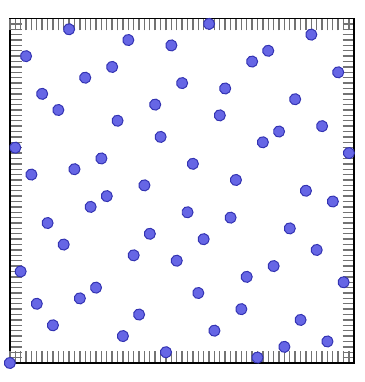}
  \caption{Left: a lattice rule with 64 points. Right: a polynomial lattice rule with 64 points.}\label{fig:lattices}
\end{figure}

\subsection{Lattice rules}

The points of a rank-$1$ \emph{lattice rule} with \emph{generating vector} $\bsz \in \Z_N^s$, with $\Z_N := \{0,\ldots,N-1\}$, are given by
\begin{align}\label{eq:lattice}
  \bsx_k
  &=
  \frac{\bsz k \bmod{N}}{N}
  ,
  &
  k = 0, \ldots, N-1
  ,
\end{align}
and its quality, for fixed $N$ and $s$, is fully determined by the choice of $\bsz$.
These point sets were introduced by Korobov \cite{Kor59} and Hlawka \cite{Hla62}; and were shown to have low discrepancy for a well chosen generating vector \cite{Kor67}.
When they minimise the error of numerical integration in some optimal way, see \RefSec{sec:ace}, they are called ``good'' lattice rules.

Often the number of points $N$ is taken to be prime to simplify proof techniques.
The lattice points form a vector space, where adding two points $\bsx_k + \bsx_\ell$ gives another point of the lattice and scalar multiplication is defined as $\ell \bsx_k = \bsx_{k \ell}$ for an integer $\ell$.
These operations can either be interpreted on the finite structure: modulo~$1$ when working with the points, or modulo~$N$ when working with the indices, as done above.
(Usually however these lattices are considered as tiling the whole space forming what is commonly understood as a ``lattice'', i.e., as a discrete infinite subset of $\R^s$.)
From the finite point of view, it makes sense to look at the range of $k$ to be numbers in $\Z_N$.

\subsection{Polynomial lattice rules}

A very similar point set, from an algebraic point of view, was introduced by Niederreiter \cite{Nie92b}, see also \cite{Nie92,DP2010}, where all the scalars in the above equation for a lattice rule\RefEq{eq:lattice} are replaced by polynomials over a finite field $\F_b[\X]$ and where $\X$ denotes the formal variable.
The lattice points are given by
\begin{align}\label{eq:poly-lattice}
  \bsx_k(\X)
  &=
  \frac{\bsz(\X) k(\X) \bmod P(\X)}{P(\X)}
  \in (\F_b((\X^{-1})))^{s}
  ,
\end{align}
for $k(\X) \in G_{b,m} := \{ k(\X) \in \F_b[\X] : \deg(k) < m \}$, 
and $\deg(0) = -\infty$ by convention.
The formal Laurent series $w(\X) \in \F_b((\X^{-1}))$ are to be understood as $w(\X) = \sum_{i \ge \ell} w_i \, \X^{-i}$ with coefficients taking values in $\F_b$.
The number of points is $N = b^m$ and, typically, $m = \deg(P)$.
Similar to above, $P(\X)$ is mostly chosen irreducible over $\F_b[\X]$.
The ``points'' are polynomials over $\F_b$ with negative powers in $\X$, or more specifically, Laurent series, where the $j$th dimension of the point $\bsx_k$ is given by
\begin{align*}
  x_{k,j}(\X)
  &=
  \sum_{i \ge 1} x_{k,j,i} \, \X^{-i}
\end{align*}
and the coefficients $x_{k,j,i} \in \F_b$ are the Laurent coefficients of the polynomial division from\RefEq{eq:poly-lattice} over the finite field $\F_b$.
Niederreiter \cite{Nie92} showed that the above structure can be mapped to real numbers in $[0,1)^s$, having the structure of a digital net in base~$b$.
For fixed $N$ and $s$, the quality of this point set is fully determined by the vector of polynomials $\bsz(\X)$.

Some more notation is needed.
The coefficients of a polynomial can be interpreted as a vector over $\F_b$ and as a scalar in $\Z_{b^m}$:
for $k(\X) \in \F_b[\X]/P(\X)$ interpret 
\begin{multline*}
  k(\X) = k_{m-1} \X^{m-1} + \cdots + k_0 \X^0
  \DG{ \\ }{}
  \quad\simeq\quad
  \vec{k} = (k_0, \ldots, k_{m-1})^{\top} \in \F_b^m
  \DG{ \\ }{}
  \quad\simeq\quad
  k = \sum_{i=0}^{m-1} \eta(k_i) \, b^i \in \Z_{b^m}
  ,
\end{multline*}
where $\eta : \F_b \to \Z_b$ is a bijection which can be the trivial map if $b$ is prime.
Similarly, a Laurent series $x(\X) \in \F_b((\X^{-1}))$ can be interpreted as a vector over $\F_b$ and as a scalar in $[0,1)$:
\begin{align*}
  x(\X) &= \sum_{i \ge 1} x_i \, \X^{-i}
  &\simeq&&
  \vec{x} &= (x_1, \ldots, x_n, \ldots)^{\top} \in \F_b^\infty
  &\simeq&&
  x = \sum_{i \ge 1} \theta(x_i) \, b^{-i}
  ,
\end{align*}
again with a bijection $\theta : \F_b \to \Z_b$.
To simplify the presentation it will be assumed that $b$ is prime and thus the mapping in both cases can be taken as the canonical map.
Sometimes below, for a positive integer $M$, this infinite expansion will be truncated at $\X^{-M}$, or $\X^M$ when considering arbitrarily large integers, and the truncated versions will be denoted by $[x(\X)]_M = \sum_{i=1}^M x_i \, \X^{-i}$, $[\vec{x}]_M = (x_1, \ldots, x_M)^{\top} \in \F_b^M$ and $[x]_M = \sum_{i=1}^M x_i \, b^{-i}$.

For prime $b$ the polynomial lattice points\RefEq{eq:poly-lattice} can be mapped to the unit cube by ``evaluating'' the polynomial point up to some precision $n$, typically $n=m$, i.e.,
\begin{align*}
  y_{k,j}
  =
  [ x_{k,j} ]_n
  =
  [ x_{k,j}(b) ]_n
  &=
  \sum_{i = 1}^n x_{k,j,i} \, b^{-i}
  .
\end{align*}
The rule\RefEq{eq:QN} using the points $\{ \bsy_k \}_{k=0}^{b^m-1}$, where $\bsy_k = (y_{k,1},\ldots,y_{k,s})$, is called a (rank-$1$) \emph{polynomial lattice rule}.

The process described above is equivalent to the \emph{digital construction scheme} from Niederreiter, see, e.g., \cite{Nie92,DP2010}, using specific \emph{generating matrices} $C_j(z_j, P) = C_j \in \F_b^{n \times m}$ where the matrix $C_j = ( c_{j,r,t} )_{1 \le r \le n, 1 \le t \le m}$ is given by $c_{j,r,t} = a_{j,r+t-1}$ and the $a_{j,i}$ are the coefficients of $a_j(\X) = z_j(\X) / P(\X) = \sum_{i\ge1} a_{j,i} \, \X^{-i} \in \F_b(\X^{-1})$.
The points are then generated by
\begin{align*}
  \vec{y}_{k,j}
  =
  [\vec{x}_{k,j}]_n
  &=
  C_j \, \vec{k}
  ,
\end{align*}
with $\vec{k} \in \F_b^m$.

For \emph{higher-order polynomial lattice rules} the precision is taken to be $n = \alpha m$ where the integer $\alpha \ge 1$ will denote the order of convergence $O(N^{-\alpha+\epsilon})$, $\epsilon > 0$. 

\section{The worst-case error}\label{sec:ace}

The idea now is to find ``optimal'' generating vectors.
For this, define the error of approximating the integral by a lattice rule or a polynomial lattice rule $Q_N(\cdot; \bsz)$,
\begin{align*}
  E_N(f; \bsz)
  &:=
  Q_N(f; \bsz) - I(f)
  .
\end{align*}
Note that the dependency on $\bsz(\X)$, $P(\X)$ and $N=b^m$ for the polynomial lattice rule is suppressed by referring to just $\bsz$ and $N$ as in the lattice rule case.
By assuming certain properties on the function $f$ the quantity $|E(f; \bsz)|$ can be bounded from above (and below) with respect to the worst possible function satisfying the assumed properties.
An upper bound can be obtained by applying Hölder's inequality, as will be shown in \RefThm{thm:KH}.
A similar exposition of this analysis can be found in \cite{Hic98a,Hic98}.

\subsection{Koksma--Hlawka error bound}

The assumptions on $f$ are formalized by a weighted $\ell_p$ semi-norm $|||f|||_{p,\alpha,\bsgamma} \le 1$, where $\alpha$ expresses how quickly a certain series expansion of $f$ converges and $\bsgamma = \{ \gamma_\setu \}_{\setu \subseteq \{1:s\}}$ is a set of non-negative weights which determine the influence of the different dimensions.
The notation $\{1:s\}$ is a shorthand for $\{1,\ldots,s\}$.
For classical spaces, i.e., unweighted spaces, all weights are unity, $\gamma_\setu \equiv 1$.
A discussion on the weights is deferred until \RefSec{sec:weights}.
The \emph{worst-case error} of integrating such $f$, with $|||f|||_{p,\alpha,\bsgamma} \le 1$, can now be defined as
\begin{align}\label{eq:wce}
  e(\bsz, N; |||\cdot|||_{p,\alpha,\bsgamma})
  &:=
  \sup_{|||f|||_{p,\alpha,\bsgamma} \le 1} |Q_N(f; \bsz) - I(f)|
  .
\end{align}
The following theorem shows an upper bound for functions expressed in a certain basis $\{ \varphi_{\bsh} \}_{\bsh \in \Lambda}$.
This basis is mostly the Fourier basis, with $\varphi_{\bsh}(\bsx) = \exp(\twopii \bsh \cdot \bsx)$ and $\Lambda = \Z^s$, when discussing lattice rules and the Walsh basis, with $\varphi_{\bsh}(\bsx) = \wal_{b,\bsh}(\bsx)$ and $\Lambda = \N_0^s = \{0,1,\ldots\}^s$, when discussing polynomial lattice rules.
The details will follow in \RefSecTwo{sec:lattice-rules}{sec:poly-lattice-rules}.
\begin{theorem}[Koksma--Hlawka error bound]\label{thm:KH}
  Suppose $f$ can be expressed as an absolutely convergent series by a basis $\{ \varphi_{\bsh} \}_{\bsh \in \Lambda}$, 
  $$
    f(\bsx)
    =
    \sum_{\bsh \in \Lambda} \hat{f}_{\bsh} \, \varphi_{\bsh}(\bsx)
    ,
  $$
  for an index set $\Lambda \subseteq \Z^s$, $\bszero \in \Lambda$, with $\varphi_{\bszero} = 1$ and such that $\hat{f}_{\bszero} = \int_{[0,1)^s} f(\bsx) \rd\bsx = I(f)$.
  Then
  \begin{align}\label{eq:error}
    E_N(f; \bsz)
    &=
    \sum_{\bszero \ne \bsh \in \Lambda} \hat{f}_{\bsh} \frac1N \sum_{k=0}^{N-1} \varphi_{\bsh}(\bsx_k)
    .
  \end{align}
  Furthermore assume a real valued function, the ``decay function'', $r_{\alpha,\bsgamma}(\bsh) > 0$ for all $\bsh \in \Lambda$, and, for $1 \le p < \infty$, define
  \begin{align*}
    ||| f |||_{p,\alpha,\bsgamma}^p
    &:=
    \sum_{\bszero \ne \bsh \in \Lambda} |\hat{f}_{\bsh}|^p \, r_{\alpha,\bsgamma}(\bsh)^p
    ,
  \end{align*}
  and for $p = \infty$ define
  \begin{align*}
    ||| f |||_{\infty,\alpha,\bsgamma}
    &:=
    \sup_{\bszero \ne \bsh \in \Lambda} |\hat{f}_{\bsh}| \, r_{\alpha,\bsgamma}(\bsh)
    .
  \end{align*}
  Then for $\tfrac1p + \tfrac1q = 1$ 
  \begin{align*}
    E_N(f; \bsz)
    &\le
    ||| f |||_{p,\alpha,\bsgamma} \;
    e(\bsz, N; |||\cdot|||_{p,\alpha,\bsgamma})
  \end{align*}
  where for $1 < p \le \infty$
  \begin{align}\label{eq:wce-p}
    e(\bsz, N; |||\cdot|||_{p,\alpha,\bsgamma})
    &=
    \left(\sum_{\bszero \ne \bsh \in \Lambda}  r_{\alpha,\bsgamma}(\bsh)^{-q} \left| \frac1N \sum_{k=0}^{N-1} \varphi_{\bsh}(\bsx_k) \right|^q\right)^{1/q}
    ,
  \end{align}
  or for $p=1$ and $q=\infty$
  \begin{align}\label{eq:wce-infty}
    e(\bsz, N; |||\cdot|||_{1,\alpha,\bsgamma})
    &=
    \sup_{\bszero \ne \bsh \in \Lambda}  r_{\alpha,\bsgamma}(\bsh)^{-1} \left| \frac1N \sum_{k=0}^{N-1} \varphi_{\bsh}(\bsx_k) \right|
    .
  \end{align}
\end{theorem}
\begin{proof}
  Straightforwardly, as the series converges absolutely and point wise, $f$ can be expanded and the sums interchanged
  \begin{align}\notag
    Q_N(f; \bsz) - I(f)
    &=
    \frac1N \sum_{k=0}^{N-1} \sum_{\bsh \in \Lambda} \hat{f}_{\bsh} \, \varphi_{\bsh}(\bsx_k) - I(f)
    \\\label{eq:holder}
    &=
    \sum_{\bszero \ne \bsh \in \Lambda} \hat{f}_{\bsh} \, r_{\alpha,\bsgamma}(\bsh) \, r_{\alpha,\bsgamma}(\bsh)^{-1} \frac1N \sum_{k=0}^{N-1} \varphi_{\bsh}(\bsx_k)
    ,
  \end{align}
  from which the results follow by applying the H\"older inequality.
\end{proof}

Some remarks are in order.
In the above theorem the decay function $r_{\alpha,\bsgamma}(\bsh)$ is there to control how quickly the series representation converges.
Obviously, the theorem only makes sense whenever $|||f|||_{p,\alpha,\bsgamma} < \infty$ and this condition then defines a function space for which the worst-case error rate holds.
The case $r_{\alpha,\bsgamma}(\bsh) = 0$ for some $\bsh$ is allowed by excluding those $\bsh$ from the index set $\Lambda$.
The sum
\begin{align}\label{eq:sum-basis-funs}
  \frac1N \sum_{k=0}^{N-1} \varphi_{\bsh}(\bsx_k)
  ,
\end{align}
which appearaned in\RefEqTwo{eq:wce-p}{eq:wce-infty}, plays a crucial role and will be of use in the next subsections.
All results will be presented for equal weight cubature rules\RefEq{eq:QN}, but similar results, more involved, can be derived for more general cubature rules $Q_N(f) = \sum_{k=0}^{N-1} w_k f(\bsx_k)$.
E.g., if $\sum_{k=0}^{N-1} w_k= 1$ then the above theorem holds with\RefEq{eq:sum-basis-funs} replaced by
\begin{align*}
  \sum_{k=0}^{N-1} w_k \, \varphi_{\bsh}(\bsx_k)
  .
\end{align*}
In what follows the set of functions $\{\varphi_{\bsh}\}_{\bsh \in \Lambda}$ will be mostly the Fourier basis for lattice rules, but also the cosine basis in \RefSec{sec:tent}, and the Walsh basis for polynomial lattice rules.
For more arbitrary $\varphi_{\bsh}$ it is assumed that if $h_j = 0$ then $\varphi_{\bsh}(\bsx)$ is independent of $x_j$, which is denoted as ``zero-neutral'' in this manuscript; otherwise the weights will not make sense as the weight $\gamma_{\setu}$ is only supposed to model the influence of the variables $\{x_j\}_{j \in \setu}$.
The most natural form is a product basis $\varphi_{\bsh}(\bsx) = \prod_{j=1}^s \varphi^{(j)}_{h_j}(x_j)$ with $\varphi^{(j)}_{0} = 1$, which automatically fulfills this assumption.
Similarly also the decay function $r_{\alpha,\bsgamma}(\bsh)$ should be ``zero neutral'' for the same reason.

\subsection{Lattice rules}\label{sec:lattice-rules}

For lattice rules it turns out to be convenient to work with some kind of Fourier space, i.e., a function space based on Fourier series expansions
\begin{align}\label{eq:fourier-series}
  f(\bsx)
  &=
  \sum_{\bsh \in \Z^s} \hat{f}_{\bsh} \, \E^{\twopii \bsh \cdot \bsx}
  ,
  &
  \text{where }
  \hat{f}_{\bsh}
  &=
  \int_{[0,1)^s} f(\bsx) \, \E^{-\twopii \bsh \cdot \bsx} \rd\bsx
  .
\end{align}
The sum\RefEq{eq:sum-basis-funs} then reduces to
\begin{align}\label{eq:character-prop}
  \frac1N \sum_{k=0}^{N-1} \E^{\twopii (h_1 z_1 + \cdots + h_s z_s) \, k / N}
  &=
  \begin{cases}
    1 & \text{if } h_1 z_1 + \cdots + h_s z_s \equiv 0 \pmod{N} , \\
    0 & \text{otherwise}.
  \end{cases}
\end{align}
This is known as the \emph{character property} of $\Z_N$.
Note the following similar one-dimensional sum based on the character property, where for $N$ prime and $0 \le k < N$
\begin{align}\label{eq:sum-over-ZN}
  \frac1N \sum_{z = 0}^{N-1} \E^{\twopii h z k / N}
  &=
  \begin{cases}
    1 & \text{if } k = 0 \text{ or } h \equiv 0 \pmod{N} , \\
    0 & \text{otherwise}.
  \end{cases}
\end{align}
This will become of use in \RefThm{thm:cbc} on the component-by-component construction.
Those $\bsh \in \Z^s$ which fulfill the condition $\bsh \cdot \bsz \equiv 0 \pmod{N}$ from\RefEq{eq:character-prop} are elements of the \emph{dual lattice} denoted by $L^\perp = L^\perp(\bsz, N)$.
Thus
\begin{align*}
  E_N(f; \bsz)
  &=
  \sum_{\bszero \ne \bsh \in L^\perp} \hat{f}_{\bsh}
  .
\end{align*}
It follows that the worst-case error, for a lattice rule in a Fourier space, is given by, with $1 < p \le \infty$,
\begin{align}\label{eq:wce-lattice-rule}
  e(\bsz, N; |||\cdot|||_{p,\alpha,\bsgamma})
  &=
  \left( \sum_{\bszero \ne \bsh \in L^\perp} r_{\alpha,\bsgamma}(\bsh)^{-q} \right)^{1/q}
  ,
\end{align}
where $q=p/(p-1)$, and for $p=1$ and $q=\infty$,
\begin{align}\label{eq:wce-lattice-rule-infty}
  e(\bsz, N; |||\cdot|||_{1,\alpha,\bsgamma})
  &=
  \sup_{\bszero \ne \bsh \in L^\perp} r_{\alpha,\bsgamma}(\bsh)^{-1}
  .
\end{align}

Choosing algebraically decaying Fourier modes on hyperbolic cross isolines (also called Zaremba crosses), $$r_{\alpha,\bsgamma}(\bsh) = \prod_{j=1}^s \max(1, |h_j|)^{\alpha},$$ where all $\gamma_{\setu} \equiv 1$, then results in the classical \emph{Korobov class} of functions \cite{Kor59,Kor60} when $p=\infty$.
The bound from above for the case $p = 2$ and $q = 2$, ``the Hilbert case'', which is studied often in current literature in the form of \emph{reproducing kernel Hilbert spaces}, see, e.g., \cite{DKS2013} for a recent overview,  gives
\begin{align*}
  E_N(f; \bsz)
  &\le
  \left( \sum_{\bszero \ne \bsh \in \Z^s} |\hat{f}_{\bsh}|^2 \prod_{j=1}^s \max(1, |h_j|)^{2\alpha} \right)^{\!\!1/2} 
  \DG{\hspace{-1mm}}{}
  \left( \sum_{\bszero \ne \bsh \in L^\perp} \prod_{j=1}^s \max(1, |h_j|)^{-2\alpha}  \right)^{\!\!1/2}
  .
\end{align*}
It is interesting to compare this to the bound for $p=\infty$ and $q=1$
\begin{align*}
  E_N(f; \bsz)
  &\le
  \left( \sup_{\bszero \ne \bsh \in \Z^s} |\hat{f}_{\bsh}| \, \prod_{j=1}^s \max(1, |h_j|)^\alpha \right)
  \left( \sum_{\bszero \ne \bsh \in L^\perp} \prod_{j=1}^s \max(1, |h_j|)^{-\alpha} \right)
  ,
\end{align*}
where the classic quantity $P_\alpha$ occurs,
\begin{align*}
  P_\alpha(\bsz, N)
  &:=
  \sum_{\bszero \ne \bsh \in L^\perp} \prod_{j=1}^s \max(1, |h_j|)^{-\alpha}
  ,
\end{align*}
as it is defined in, e.g., \cite{Nie92,SJ94}.
This is the quantity of interest in Korobov's first papers, e.g., \cite{Kor59,Kor60}.
Korobov assumes the functions to satisfy, for $\bsh \ne \bszero$,
\begin{align*}
  |\hat{f}(\bsh)|
  &\le
  c \, \prod_{j=1}^s \max(1, |h_j|)^{-\alpha}
\end{align*}
for some fixed positive constant $c$ and denotes this class by $E_\alpha^s(c)$.
This condition is equivalent to asking $|||f|||_{\infty,\alpha,\bsgamma} \le c$ for this choice of $r_{\alpha,\bsgamma}$.
Thus $P_\alpha$ is the worst-case error for the semi-norm based on the $\ell_\infty$ norm whilst for the popular $\ell_2$ case the worst-case error is given by $(P_{2\alpha}(\bsz, N))^{1/2}$.
A more general statement including weights will be given later by\RefEq{eq:Palpha-equivalence}.
From the one-dimensional case, using\RefEq{eq:wce-lattice-rule} with $N$ prime, it is clear that $\alpha > 1/q$ is required for the sums to converge, i.e., 
\begin{align}\label{eq:latrule-alpha}
  \left(\sum_{0 \ne h \in N\Z} |h|^{-q\alpha} \right)^{1/q} 
  &= 
  (2\zeta(q\alpha))^{1/q} N^{-\alpha}
  .
\end{align}
So, to keep $\zeta(q\alpha) < \infty$, it is needed that $\alpha > 1/2$ for the $\ell_2$ case and $\alpha > 1$ for $\ell_\infty$, i.e., the class $E_{\alpha}^s(c)$.

The bounds from above are all attainable.
For a rank-$1$ lattice rule, and $1 < p \le \infty$, take the function
\begin{align}\label{eq:wci}
  \xi(\bsx)
  =
  \xi(\bsx; \bsz, N, |||\cdot|||_{p,\alpha,\bsgamma})
  &=
  \sum_{\bszero \ne \bsh \in L^\perp}
  r_{\alpha,\bsgamma}(\bsh)^{-q}
  \, \E^{\twopii \bsh \cdot \bsx}
  ,
\end{align}
which depends on the point set (i.e., $\bsz$ and $N$), the smoothness $\alpha$, the weights $\bsgamma = \{ \gamma_{\setu} \}_{\setu \subseteq \{1:s\}}$ and the choice of $p$ and $q$,
and which has semi-norm, for $1 < p < \infty$, 
\begin{multline*}
  |||\xi|||_{p,\alpha,\bsgamma}^p
  =
  \sum_{\bszero \ne \bsh \in L^\perp}
  r_{\alpha,\bsgamma}(\bsh)^{-p q} \, r_{\alpha,\bsgamma}(\bsh)^{p}
  \\=
  \sum_{\bszero \ne \bsh \in L^\perp}
  r_{\alpha,\bsgamma}(\bsh)^{-p(q-1)}
  =
  \sum_{\bszero \ne \bsh \in L^\perp}
  r_{\alpha,\bsgamma}(\bsh)^{-q}
  ,
\end{multline*}
and $|||\xi|||_{\infty,\alpha,\bsgamma} = 1$ for $p = \infty$,
and thus, since $\tfrac1p + \tfrac1q = 1$, the Hölder inequality is turned into an equality, as, for $1 < p < \infty$,
\begin{multline*}
  \left(\sum_{\bszero \ne \bsh \in L^\perp}
  r_{\alpha,\bsgamma}(\bsh)^{-q}\right)^{1/p}
  \left(\sum_{\bszero \ne \bsh \in L^\perp} r_{\alpha,\bsgamma}(\bsh)^{-q}\right)^{1/q}
  \\=
  \sum_{\bszero \ne \bsh \in L^\perp}
  r_{\alpha,\bsgamma}(\bsh)^{-q}
  =
  E_N(\xi; \bsz)
  =
  \left(e(\bsz, N; |||\cdot|||_{p,\alpha,\bsgamma})\right)^q
  .
\end{multline*}
Scaling $\xi$ by $|||\xi|||_{p,\alpha,\bsgamma}^{-1}$ then gives a function with semi-norm~1 and error exactly equal to the worst-case error.
For $p=\infty$ and $q=1$ the worst-case error bound is clearly an equality as then $|||\xi|||_{\infty,\alpha,\bsgamma} = 1$.

For calculating the worst-case error, the function $\xi(\bsx; \bsz, N, |||\cdot|||_{p,\alpha,\bsgamma})$ for all $1 < p \le \infty$ can be replaced by a function which does not depend on the point set but only on the function space, being $p$, $\alpha$ and $\bsgamma$, namely, 
\begin{align}\label{eq:chi}
  \chi(\bsx; |||\cdot|||_{p,\alpha,\bsgamma})
  &=
  \sum_{\bszero \ne \bsh \in \Z^s}
  r_{\alpha,\bsgamma}(\bsh)^{-q}
  \, \E^{\twopii \bsh \cdot \bsx}
  ,
\end{align}
since the errors are the same, i.e., $E_N(\xi; \bsz) = E_N(\chi; \bsz)$.
This property can be used to calculate the worst-case error for $1 < p \le \infty$ and will be of use in \RefSec{sec:construction}.

Now consider $p=1$ and $q=\infty$.
The function $\xi$ can here be constructed by choosing any $\bsh^\star \in L^\perp$ for which $\sup_{\bszero \ne \bsh \in L^\perp} r_{\alpha,\bsgamma}(\bsh)^{-1} = r_{\alpha,\bsgamma}(\bsh^\star)^{-1}$.
There are always at least two choices of $\bsh^\star$ as, if $\bsh^\star \in L^\perp$, then so is $-\bsh^\star$.
The function
\begin{align*}
  \xi(\bsx; \bsz, N, |||\cdot|||_{1,\alpha,\bsgamma})
  &=
  r_{\alpha,\bsgamma}(\bsh^\star)^{-1}
  \, \E^{\twopii \bsh^\star \cdot \bsx}
\end{align*}
then has semi-norm $|||\xi|||_{1,\alpha,\bsgamma} = 1$.
The error $E_N(\xi; \bsz, N, |||\cdot|||_{1,\alpha,\bsgamma}) = r_{\alpha,\bsgamma}(\bsh^\star)^{-1}$  equals the worst-case error for $q=\infty$ by definition.
There is however no function $\chi$ which is independent of the point set as there is for the other choices of $p$ and $q$.
This means that the worst-case error cannot be computed in a comfortable way for $q=\infty$.
Simply iterating over $\bszero \ne \bsh \in L^\perp$ ordered on $r_{\alpha,\bsgamma}(\bsh)^{-1}$ to find the first $\bsh^\star$ for fixed $\bsz$ and $N$ has exponential complexity for most classical choices of $r_{\alpha,\bsgamma}$, see, e.g., \cite{CKN2010}, and also \cite{CN2008}, and the references therein for alternative strategies.

%\begin{figure}
%  \centering
%  \begin{tikzpicture}
%     \draw (0,0) ellipse (3cm and 1.5cm) node[right=3.25cm,above=3mm]{$B_\infty$};
%     \draw (-0.25,0) ellipse (2cm and 1cm) node[right=2.15cm,above=2mm]{$B_p$};
%     \draw (-0.5,0) ellipse (1cm and 0.5cm) node[right=1.2cm,above=1mm]{$B_1$};
%  \end{tikzpicture}
%  \caption{Embedding of unit balls for the worst-case error with respect to different semi-norms $|||\cdot|||_p$, $1 \le p \le \infty$.}\label{fig:unitballs}
%\end{figure}

For the choice of $r_{\alpha,\bsgamma}$ as given above, the worst-case error for $q=\infty$ is directly related to the \emph{Zaremba index}, or Zaremba figure of merit, see, e.g., \cite{SJ94,Nie92}, which is defined, with all $\gamma_\setu \equiv 1$, as
\begin{align}\label{eq:Zaremba-index}
  \rho(\bsz, N)
  &:=
  \min_{\bszero \ne \bsh \in L^\perp} r_{1,\bsgamma}(\bsh)
  ,
\end{align}
and for which larger values denote better lattice rules.
The worst-case error is then given by $\rho(\bsz, N)^{-\alpha}$.
Such figures of merit are related to the classical concept of \emph{degree of precision} which has been studied in, e.g., \cite{CKN2010,AN2012}.
It can be seen that smaller $p$ will shrink the unit ball on which the worst-case error\RefEq{eq:wce} is defined, since $|||f|||_{r,\alpha,\bsgamma} \ge |||f|||_{r',\alpha,\bsgamma}$ for any $1 \le r \le r' \le \infty$.
%Denote the unit ball for $1 \le r \le \infty$ by $B_r := \{ f : |||f|||_r \le 1 \}$, then \RefFig{fig:unitballs} depicts this situation.
The case $p=1$ can be considered as the limit of $r \to 1$, which means $q \to \infty$ and then the series expansion\RefEq{eq:wci} needs to converge faster than at an algebraic rate in the limit.
This naturally leads to the recently studied exponentially converging function spaces as in \cite{DLPW2011,KPW2013}.
Similarly, the method in \cite{AN2012} is based on exponentially converging series to construct lattice rules with good \emph{trigonometric degree}.

It follows that, for $1 \le p \le \infty$,
\begin{align*}
  e(\bsz, N; |||\cdot|||_{1,\alpha,\bsgamma})
  \le
  e(\bsz, N; |||\cdot|||_{p,\alpha,\bsgamma})
  \le
  e(\bsz, N; |||\cdot|||_{\infty,\alpha,\bsgamma})
  ,
\end{align*}
and so upper bounds for $\ell_\infty$ also hold for smaller $p$ and lower bounds for $\ell_1$ also hold for larger $p$, see also \cite{Woz2009}, which states that ``multivariate  integration over $F_\infty$ is no easier than over $F_2$'' (where $F_\infty$ refers to the space using the $\ell_\infty$ norm and likewise for $F_2$).
Because of the rather different nature of the case $p=1$ and $q=\infty$, the construction part will only discuss $1 < p \le \infty$.

\subsection{Polynomial lattice rules}\label{sec:poly-lattice-rules}

Like Fourier series work naturally with lattice rules, so do Walsh series (in base~$b$) for polynomial lattice rules (in base~$b$).
Define $\N := \{1,2,\ldots\}$ and $\N_0 := \{0,1,2,\ldots\}$.
The one-dimensional Walsh functions in base~$b$ are defined as
\begin{align*}
  \wal_{b,h}(x)
  &:=
  \E^{\twopii (x_1 h_0 + x_2 h_1 + \cdots + x_n h_{n-1}) / b}
  =
  \E^{\twopii [\vec{h}]_n^\top [\vec{x}]_n / b}
\end{align*}
for $x \in [0,1)$ and $h \in \N_0$ and the unique base~$b$ expansions $x = \sum_{i\ge1} x_i \, b^{-i} = (0.x_1 x_2 \ldots)_b$ and $h = \sum_{i\ge0} h_i \, b^{i} = (h_{n-1} \cdots h_1 h_0)_b$, with $n$ at least as large as the number of digits to represent $x$ or $h$.
Multivariate Walsh functions are defined as the product of the one-dimensional Walsh functions
\begin{align*}
  \wal_{b,\bsh}(\bsx)
  &:=
  \prod_{j=1}^s \wal_{b,h_j}(x_j)
  .
\end{align*}
The Walsh functions span $L_2([0,1)^s)$ \cite{Wal23,Fin49}.
Now consider $f$ expanded in its Walsh series in base~$b$
\begin{align*}
  f(\bsx)
  &=
  \sum_{\bsh \in \N_0^s} \hat{f}_{\bsh} \, \wal_{b,\bsh}(\bsx)
  ,
  &
  \text{where }
  \hat{f}_{\bsh}
  =
  \hat{f}_{b,\bsh}
  &=
  \int_{[0,1)^s} f(\bsx) \, \overline{\wal_{b,\bsh}(\bsx)} \rd\bsx
  .
\end{align*}
The sum\RefEq{eq:sum-basis-funs} can then be written, making use of the generating matrices $C_j \in \F_b^{n \times m}$ and setting $\vec{w} := \sum_{j=1}^s  C_j^\top [\vec{h}_j]_n \in \F_b^m$.
Then
\begin{align}\notag
  \frac1{b^m} \sum_{k=0}^{b^m-1} \prod_{j=1}^s \E^{\twopii [\vec{h}_j]_n^\top C_j \vec{k} / b}
  &=
  \frac1{b^m} \sum_{k=0}^{b^m-1} \E^{\twopii \vec{k}^\top (\sum_{j=1}^s  C_j^\top [\vec{h}_j]_n) / b} 
  \\\notag
  &=
  \prod_{i=0}^{m-1} \frac1b \sum_{k_i \in \F_b} \E^{\twopii k_i w_{i+1} / b}
  =
  \prod_{i=0}^{m-1}
  \begin{cases}
    1 & \text{if } w_{i+1} = 0 \in \F_b , \\
    0 & \text{otherwise} ,
  \end{cases}
  \\\label{eq:poly-character-prop}
  &=
  \begin{cases}
    1 & \text{if } \vec{w} = \sum_{j=1}^s  C_j^\top [\vec{h}_j]_n = \vec{0} \in \F_b^m, \\
    0 & \text{otherwise} .
  \end{cases}
\end{align}
This is the character property for polynomial lattice rules.
Similar to\RefEq{eq:sum-over-ZN}, for $0 \le k < b^m$, $P(\X)$ irreducible and $y_k(z,P) \in [0,1)$ the $k$th polynomial lattice point for a generator polynomial $z(\X)$ modulo $P(\X)$, 
\begin{align}\notag
  \frac1{b^n} \sum_{z \in G_{b,n}} \wal_{b,h}(y_k(z,P))
  &=
  \frac1{b^n} \sum_{z \in G_{b,n}} \E^{\twopii [\vec{h}]_n^\top C(z, P) \, \vec{k} \, / b}
  \\\notag
  &=
  \begin{cases}
    \prod_{i=1}^n \frac1b \sum_{y_i \in \F_b} \E^{\twopii h_{i-1} y_i / b} & \text{if } k\ne0, \\
    1 & \text{if } k = 0 ,
  \end{cases}
  \\\label{eq:sum-over-Gbn}
  &=
  \begin{cases}
    1 & \text{if }  k = 0 \text{ or } h \equiv 0 \pmod{b^n} , \\
    0 & \text{otherwise}.
  \end{cases}
\end{align}
The equivalent observation is that all one-dimensional points $(0.y_1 \ldots y_n)_b$ are generated by looping over all possible generator polynomials $z(\X)$ and keeping $k$ fixed, but different from $0$, when $P(\X)$ is an irreducible polynomial over $\F_b$.
Again, this will become of use in \RefThm{thm:cbc} on the component-by-component construction, see also \cite{KP2007} for non-irreducible $P(\X)$.

Condition\RefEq{eq:poly-character-prop} defines the dual lattice of the polynomial lattice rule:
\begin{align*}
  L^\perp
  &=
  \left\{ \bsh \in \N_0^s : \sum_{j=1}^s  C_j^\top [\vec{h}_j]_n = \vec{0} \in \F_b^m \right\}
  .
\end{align*}
With some more work it can be formulated into a polynomial version as is shown next.
\begin{lemma}
  The dual of a polynomial lattice rule with $b^m$ points and generating vector $\bsz \in G_{b,n}^s$ modulo $P(\X) \in \F_b[\X]$, with $\deg(P) = n \ge m$, is given by
\begin{align*}
  L^\perp
  &=
  \left\{ \bsh \in \N_0^s : \sum_{j=1}^s z_j(\X) [h_j(\X)]_n \equiv a(\X) \; (\operatorname{mod}{P(\X)}) \text{ for which } \deg(a) < n - m \right\}
  .
\end{align*}
\end{lemma}
\begin{proof}
  To find $\sum_{j=1}^s  C_j^\top [\vec{h}_j]_n = \vec{0}$ define
  $\vec{w}_j = C_j^\top [\vec{h}_j]_n$ such that $\sum_{j=1}^s \vec{w}_j = \vec{0}$ is needed to complete the proof.
  The product $\vec{w}_j = C_j^\top [\vec{h}_j]_n$ corresponds to the matrix-vector product
  \begin{align*}
    \begin{pmatrix}
      a_{j,1} & a_{j,2} & \cdots & a_{j,n} \\
      a_{j,2} & a_{j,3} & \cdots & a_{j,n+1} \\
      \vdots & \vdots & & \vdots \\
      a_{j,m} & a_{j,m+1} & \cdots & a_{j,m+n-1}
    \end{pmatrix}
    \begin{pmatrix}
      h_{j,0} \\ h_{j,1} \\ h_{j,2} \\ \vdots \\ h_{j,n-1}
    \end{pmatrix}
    &=
    \begin{pmatrix}
      w_{j,1} \\ w_{j,2} \\ \vdots \\ w_{j,m}
    \end{pmatrix}
  \end{align*}
  over $\F_b$.
  This matrix-vector product can be interpreted as a finite precision version of $w_j(\X) := a_j(\X) \, [h_j(\X)]_n \bmod{1(\X)}$ where $a_j(\X) = z_j(\X) / P(\X) \in \F_b((\X^{-1}))$ and $w_j(\X) \in \F_b((\X^{-1}))$; and $\vec{w}_j \in \Z_b^m$ is obtained from the truncation $[w_j(\X)]_m$.
  Now consider the infinite precision polynomial sum
  \begin{multline*}
    \sum_{j=1}^s w_j(\X)
    \DG{ \\ }{}
    =
    \sum_{j=1}^s
    \frac{z_j(\X)}{P(\X)} \, [h_j(\X)]_n \bmod{1(\X)}
    =
    0 \, \X^{-1} + \cdots + 0 \, \X^{-m} + w_{m+1} \, \X^{-(m+1)} + \cdots
    \\
    \Leftrightarrow \qquad
    \sum_{j=1}^s
    z_j(\X) \, [h_j(\X)]_n
    \equiv
    a(\X)
    \pmod{P(\X)}
    ,
  \end{multline*}
  for any $a(\X) \in \F_b[\X]$ with $\deg(a) < n-m$.
\end{proof}

Thus, also here, in terms of Walsh coefficients in the same base, the error can be expressed as the sum of the Walsh coefficients in the dual:
\begin{align*}
  E_N(f; \bsz)
  &=
  \sum_{\bszero \ne \bsh \in L^\perp} \hat{f}_{\bsh}
  .
\end{align*}
Hence also the worst-case error takes exactly the forms\RefEqTwo{eq:wce-lattice-rule}{eq:wce-lattice-rule-infty} as for a lattice rule.
That is, for $1 < p \le \infty$,
\begin{align*}
  e(\bsz, N; |||\cdot|||_{p,\alpha,\bsgamma})
  &=
  \left( \sum_{\bszero \ne \bsh \in L^\perp} r_{\alpha,\bsgamma}(\bsh)^{-q} \right)^{1/q}
  ,
\end{align*}
with $q=p/(p-1)$, and for $p=1$ and $q=\infty$,
\begin{align*}
  e(\bsz, N; |||\cdot|||_{1,\alpha,\bsgamma})
  &=
  \sup_{\bszero \ne \bsh \in L^\perp} r_{\alpha,\bsgamma}(\bsh)^{-1}
  .
\end{align*}
Because of this similarity and the shorthand notation just referring to $\bsz$ and $N$, large parts of \RefSec{sec:lattice-rules} can be transplanted to a polynomial lattice rule equivalent.
E.g., assuming algebraically decaying Walsh modes,
\begin{align*}
  r_{\alpha,\bsgamma}(\bsh)
  &=
  \prod_{j=1}^s b^{\alpha \floor{\log_b h_j}}
  ,
\end{align*}
where $\log_b$ is the logarithm in base~$b$ (with the convention
$\log_b 0 = -\infty$), leads to a so-called \emph{Walsh space}, see, e.g., \cite{DKPS2005,DP2005}.
Similar to\RefEq{eq:latrule-alpha} it can be shown that $\alpha > 1/q$ for this setting.

Walsh series in base~$2$ are equivalent to standard Haar series.
However, for continuous, one-dimensional $f$ it is known that if its Haar coefficients decay faster than $h^{-3/2}$ with respect to the orthonormal Haar basis then $f$ is constant on $[0,1]$, and a similar remark is made with respect to Walsh series \cite{Fin49}.
This means the decay can only be moderate with such a choice of $r_{\alpha,\bsgamma}$.
(See also \cite{Nuy2007} for a reproducing kernel Hilbert space based on Haar wavelets and the equivalence with the Walsh space.)
The power of the Walsh series however lies in more complicated functions $r_{\alpha,\bsgamma}$ such that certain Sobolev spaces are embedded in it, see, e.g., \cite{DP2005,DP2007,Dic2008,DP2010}.
This will be made more explicit in \RefSec{sec:poly-spaces}.

Like the Zaremba index\RefEq{eq:Zaremba-index} for lattice rules, a similar figure of merit can be defined for polynomial lattice rules:
\begin{align*}
  \rho(\bsz(\X), P(\X))
  &:=
  (s-1) + \! \min_{\bszero \ne \bsh \in L^\perp} \sum_{j=1}^s \deg(h_j(\X))
  =
  (s-1) + \! \min_{\bszero \ne \bsh \in L^\perp}  \log_b r_{1,\bsgamma}(\bsh)
  ,
\end{align*}
where for the last equality $\gamma_\setu \equiv 1$,
see, e.g., \cite{Nie92,DP2010,Pil2012}.
(In the same references it is shown that a polynomial lattice rule is a strict $(t,m,s)$-net in base~$b$ with $t = m - \rho(\bsz(\X), P(\X))$. See \cite{Nie92,DP2010,Pil2012} for $(t,m,s)$-nets, $t$-values and their relationship to Walsh functions.)

Exactly the same remark about the special case for $p=1$ holds here as well as it is independent of the specifics of the function space and so only $1 < p \le \infty$ will be considered in what follows.
In this case also for polynomial lattice rules there is a function $\chi$ for which $E_N(\chi;\bsz)^{1/q} = e(\bsz, N; |||\cdot|||_{p,\alpha,\bsgamma})$ and it is given by
\begin{align}\label{eq:poly-chi}
  \chi(\bsx; |||\cdot|||_{p,\alpha,\bsgamma})
  &=
  \sum_{\bszero \ne \bsh \in \N_0^s}
  r_{\alpha,\bsgamma}(\bsh)^{-q}
  \, \wal_{b,\bsh}(\bsx)
  .
\end{align}

\section{Weighted worst-case errors}\label{sec:weights}

In \RefThm{thm:KH}, the decay function $r_{\alpha,\bsgamma}(\bsh)$ controls the convergence of the series expansion through the parameter $\alpha$ but also includes a set of $2^s$ non-negative weights $\{\bsgamma_\setu\}_{\setu \subseteq \{1:s\}}$.
These weights have been introduced to control the dependence on the number of dimensions and to cure the curse of dimensionality.
See, e.g., \cite{SW98,Hic98,Hic98a,KS2005} and the recent monographs on tractability \cite{NW2008,NW2010,NoWo2012}.
The \emph{curse of dimensionality}, in short, means, that the worst-case error has an exponential dependency on the dimension~$s$.
The trick is now to replace the exponential dependency by a constant which can be controlled by the weights $\gamma_\setu$ and which for particular choices of weights can be bounded by an absolute constant, independent of~$s$.
For this overview it is important to take a look at the different kinds of weights which have appeared in the literature.

A natural place to introduce weights is just before applying H\"older's inequality in\RefEq{eq:holder}.
Write
\begin{align}\label{eq:r-weighted}
  r_{\alpha,\bsgamma}(\bsh)
  &=
  \gamma_{\setu(\bsh)}^{-1/2} \, r_\alpha(\bsh)
  ,
\end{align}
where the influence of $\gamma_{\setu}$ and $\alpha$ has now been separated and the ``support of $\bsh$'' is defined as
\begin{align*}
  \setu(\bsh)
  &:=
  \{ 1 \le j \le s : h_j \ne 0 \}
  .
\end{align*}
Then
\begin{align*}
  |||f|||_{p,\alpha,\bsgamma}^p
  &=
  \sum_{\bszero \ne \bsh \in \Lambda} |\hat{f}_{\bsh}|^p \, \gamma_{\setu(\bsh)}^{-p/2} \, r_{\alpha}(\bsh)^p
  ,
  \\\intertext{and}
  e(\bsz, N; |||\cdot|||_{p,\alpha,\bsgamma})^q
  &=
  \sum_{\bszero \ne \bsh \in L^\perp} \gamma_{\setu(\bsh)}^{q/2} \, r_{\alpha}(\bsh)^{-q}
  ,
\end{align*}
with the natural modification if $p=\infty$.
It would be more natural to write just $\gamma_{\setu(\bsh)}$ instead of $\sqrt{\gamma_{\setu(\bsh)}}$, however, most publications on tractability only consider the Hilbert case $p=2$ and introduce the weights directly into the inner product (or the norm); a notable exception is \cite{Was2013}.
Hence to have the results of those papers to exactly hold for the case $p=2$, the square root in\RefEq{eq:r-weighted} is retained.
A similar situation occurs with $\alpha$ versus $2\alpha$, see also \cite[Appendix~A]{NW2008}, or the introduction of ``the square of $r_{\alpha,\bsgamma}$'' in the reproducing kernel when working with reproducing kernel Hilbert spaces.

First some more notation is needed.
For $\setu \subseteq \{1:s\}$ define $$\Z_{\setu} := \{ \bsh \in \Z^s : h_j \ne 0 \text{ for } j \in \setu \text{ and } h_j = 0 \text{ for } j \not\in \setu\},$$ i.e., the support of a vector~$\bsh \in \Z_{\setu}$ is~$\setu$.
Such a vector will be denoted by $\bsh_\setu \in \Z_{\setu}$ to stress this property.
A vector $\bsh_{\setu} \in L^\perp$, for which $h_j = 0$ for $j \not\in \setu$, ignores the dimensions not in $\setu$.
E.g., for a lattice rule it follows that $\bsh_{\setu} \cdot \bsz \equiv \sum_{j\in\setu(\bsh)} h_j z_j \equiv 0 \pmod{N}$ and so there is no dependency on the dimensions of the dual not in $\setu(\bsh)$.
Therefore, also define $L^{\perp}_{\setu} := \{ \bsh \in L^{\perp} \cap \Z_{\setu} \}$.

Note that for both lattice rules and polynomial lattice rules 
\begin{align*}
  L^\perp &= \bigcup_{\setu \subseteq \{1,\ldots,s\}} L_{\setu}^\perp
  &\text{and}&&
  \Lambda &= \bigcup_{\setu \subseteq \{1,\ldots,s\}} \Lambda_{\setu}
  ,
\end{align*}
since $\Z^s = \bigcup_{\setu \subseteq \{1:s\}} \Z_{\setu}$ and $\N_0^s = \bigcup_{\setu \subseteq \{1:s\}} \N_{\setu}$,
where $\N_{\setu} := \{ \bsh \in \N_0^s : h_j \ne 0 \text{ for } j \in \setu \text{ and } h_j = 0 \text{ for } j \not\in \setu\}$.

With this notation it follows that
\begin{align*}
  e(\bsz, N; |||\cdot|||_{p,\alpha,\bsgamma})^q
  &=
  \sum_{\bszero \ne \bsh \in L^\perp} \gamma_{\setu(\bsh)}^{q/2} \, r_{\alpha}(\bsh)^{-q}
  =
  \sum_{\emptyset \ne \setu \subseteq \{1,\ldots,s\}}
  \gamma_{\setu}^{q/2}
  \sum_{\bsh_\setu \in L^{\perp}_{\setu}} r_\alpha(\bsh_\setu)^{-q}
  ,
\end{align*}
where it is assumed that $r_{\alpha}(\bsh)$ is ``zero-neutral'' in the sense that $h_j$ which are equal to zero can be ignored.

For the algebraically decaying series expansions $r_{\alpha,\bsgamma}(\bsh) = \gamma_{\setu(\bsh)}^{-1} \prod_{j \in \setu(\bsh)} |h_j|^{\alpha}$ in the case of lattice rules and for $r_{\alpha,\bsgamma}(\bsh) = \gamma_{\setu(\bsh)}^{-1} \prod_{j \in \setu(\bsh)} b^{\alpha \floor{\log_b h_j}}$ in the case of polynomial lattice rules, then, for $1 < p \le \infty$ and $\alpha > 1$, the following equality holds
\begin{align}\label{eq:Palpha-equivalence}
  e(\bsz, N; |||\cdot|||_{\infty,\alpha,\bsgamma})
  &=
  e(\bsz, N; |||\cdot|||_{p,\alpha/q,\bsgamma^{1/q}})^{p}
  ,
\end{align}
where $\bsgamma^{1/q}$ means each weight raised to the power $1/q$ and $q$ is the H\"older conjugate of~$p$.

A short overview of different types of weights is given next.
They will resurface in \RefSec{sec:fast-cbc} when the worst-case error needs to be calculated to find a good generating vector.
There are three more or less overall structural weight types:
\begin{itemize}\setlength{\itemsep}{0mm}
  \item \emph{General weights}: the term general weights is used when there is no specific structure in the weights, i.e., the $2^s$ weights are taken arbitrary, see \cite{DSWW2006}.
  \item \emph{Product weights}: with a product basis in mind, one can give a different weight $\gamma_j = \gamma_{\{j\}}$ to each dimension; the weights then take the form $\gamma_{\setu} = \prod_{j \in \setu} \gamma_j$, see \cite{SW98,Kuo2003}.
  \item \emph{Order-dependent weights}: the importance of a subset of dimensions is given by the size of the subset, that is: $\gamma_{\setu} = \Gamma_{|\setu|}$ for a set of weights $\Gamma_1$, $\ldots$, $\Gamma_s$, see \cite{DSWW2006}.
\end{itemize}
The product and order-dependent structures can also be combined:
\begin{itemize}\setlength{\itemsep}{0mm}
  \item \emph{Product-and-order-dependent weights} (POD weights): these weights are a direct combination of product weights and order-dependent weights; they take the form $\gamma_{\setu} = \Gamma_{|\setu|} \prod_{j \in \setu} \beta_j$, see~\cite{KSS2012}.
  \item \emph{Smoothness-driven product-and-order-dependent weights} (SPOD weights): they take into account norms of partial derivatives up to order $\alpha$. In \cite{DKGNS2013} such weights are derived which model $L_{\infty}$-bounds on the derivatives; they take the form 
  $\gamma_\setu = \sum_{\boldsymbol{\nu}_{\setu}}
 |\boldsymbol{\nu}_{\setu}|! \, \prod_{j\in\setu} \left(2^{\delta(\nu_j,\alpha)}\beta_j^{\nu_j}  \right)$
 where the sum is over $\boldsymbol{\nu}_{\setu} \in \{1:\alpha\}^{|\setu|}$ and $\delta(\nu_j,\alpha) = 1$ if $\nu_j = \alpha$ and zero otherwise.
\end{itemize}
There are three additional constraints which impose certain weights to be zero:
\begin{itemize}\setlength{\itemsep}{0mm}
  \item \emph{Finite-order weights}: weights are of order $q^{*}$ when $\gamma_{\setu} = 0$ for $|\setu| > q^*$, see \cite{DSWW2006}; this constraint can be combined with other structures like product and order-dependent weights, see, e.g., \cite{Gne2012}.
  \item \emph{Finite-intersection weights}: finite-intersection weights of degree $\rho$ restrict the number of overlapping sets such that for any $\setu$, for which $\gamma_{\setu} > 0$, at most $\rho$ other sets $\setv$, with weights $\gamma_{\setv} > 0$, might have a non-empty $\setu \cap \setv$. Again this constraint can be combined with other weight structures, see \cite{Gne2012}.
  \item \emph{Finite-diameter weights}: this is a subset of finite-intersection weights, and thus also of finite-order weights. Here $\gamma_{\setu} = 0$ when $\operatorname{diam}(\setu) := \max_{i,j \in \setu} |i - j| > q$, see \cite{NW2008,Gne2012}. Again this constraint can be combined with other structures.
\end{itemize}
A typical combination would be \emph{finite-order-dependent weights} which combine the order-dependent structure with the finite-order property, see \cite{DSWW2006}.

\section{Some standard spaces}\label{sec:spaces}

\RefThm{thm:KH} implies a function space for which the worst-case error bound holds by specifying $p$, a decay function $r_{\alpha,\bsgamma}$ and a basis $\{\varphi_{\bsh}\}_{\bsh}$.
The functions in this space are those for which $|||f|||_{p,\alpha,\bsgamma} < \infty$ such that the Koksma--Hlawka error bound holds.

\subsection{Lattice rules and Fourier spaces}

The example of the \emph{Korobov space} was already given above.
For the Korobov space the functions are expanded with respect to the standard Fourier basis\RefEq{eq:fourier-series}.
This results automatically in a function space of periodic functions as the series must be absolutely convergent.
The function $r_{\alpha,\bsgamma}$ here takes the form
\begin{align*}
  r_{\alpha,\bsgamma}(\bsh)
  &=
  \gamma_{\setu(\bsh)}^{-1/2}
  \prod_{j=1}^s \max(1, |h_j|)^{\alpha}
  =
  \gamma_{\setu(\bsh)}^{-1/2}
  \prod_{j \in \setu(\bsh)} |h_j|^{\alpha}
  ,
\end{align*}
with $\alpha > 1/q$.
The function $\chi$ from\RefEq{eq:chi} is given by
\begin{align}\label{eq:Korobov-omega}
  \chi(\bsx; |||\cdot|||_{p,\alpha,\bsgamma})
  &=
  \sum_{\emptyset \ne \setu \subseteq \{1:s\}}
  \gamma_{\setu}^{q/2}
  \prod_{j \in \setu} 
  \omega(x_j)
  ,
  \qquad \text{where }
  \omega(x)
  =
  \sum_{0 \ne h \in \Z} \frac{\E^{\twopii h x}}{|h|^{\alpha q}}
  ,
\end{align}
which for product weights $\gamma_{\setu(\bsh)} = \prod_{j\in \setu(\bsh)} \gamma_j$ becomes
\begin{align*}
  \chi(\bsx; |||\cdot|||_{p,\alpha,\bsgamma})
  &=
  -1 + \prod_{j=1}^s \left( 1 + \gamma_j^{q/2} \omega(x_j) \right)
  .
\end{align*}
The error, to the power $1/q$, of integrating $\chi$ then gives the worst-case error.
When $\alpha q$ is even the infinite sum for the function $\omega(x)$ above can be expressed in terms of a Bernoulli polynomial.
As there are only $N$ different values needed of this function, that is, for each one-dimensional lattice point $k/N$, $k=0,\ldots,N-1$, it can be calculated up front (and then any value of $\alpha q$ can be used).

A very similar space is the one resulting in the $R_\alpha$ criterion for $p=\infty$, see \cite{Kor67,Nie92}.
Again, functions are expressed with respect to the standard Fourier basis\RefEq{eq:fourier-series} as above, but now using a finite dimensional basis, which changes with $N$:
\begin{align}\label{eq:finite-dim}
  f(\bsx)
  &=
  \sum_{\bsh \in \left[-\tfrac N2,\tfrac N2\right)^s} \hat{f}_{\bsh} \, \E^{\twopii \bsh \cdot \bsx}
  .
\end{align}
Again the same form of $r_{\alpha,\bsgamma}$ as for the Korobov space is used, but here any $\alpha$ is fine.
The function $\chi$ looks very much like the one for the Korobov space.
E.g., for product weights it is given by
\begin{align*}
  \chi(\bsx; |||\cdot|||_{p,\alpha,\bsgamma})
  &=
  -1 + \prod_{j=1}^s \left( 1 + \gamma_j^{q/2} \omega(x_j) \right)
  ,
  \text{ where } 
  \omega(x)
  =
  \sum_{0 \ne h \in \left[-\tfrac N2,\tfrac N2\right)} \frac{\E^{\twopii h x}}{|h|^{\alpha q}}
  .
\end{align*}
The $N$ needed values $\omega(x)$ can be easily obtained by an FFT using precalculation as this sum takes exactly the form of an FFT. 
For functions like\RefEq{eq:finite-dim} it is easy to show that the worst-case error vanishes when using a regular grid with $N^s$ nodes as then there are no dual points, however, for a rank-1 lattice rule with $N$ points the number of duals in $[-N/2, N/2)^s$ is at least $N^{s-1}$.
A trivial lower bound, for $\gamma_\setu \equiv 1$, also shows that $e(\bsz, N; |||\cdot|||_{p,\alpha,\bsgamma}) \ge 2^\alpha N^{-\alpha}$ for $p \ge 1$, which is the same as for the Korobov space.
In \cite[Theorem~5.5]{Nie92}, for $p=\infty$, the value of $(R_1)^\alpha$, as well as $R_\alpha$, is used to bound $P_\alpha$ for any $\alpha > 1$.

\subsection{Randomly-shifted lattice rules and the unanchored Sobolev space}\label{sec:randomly-shifted}

By adding a (random) shift $\bsDelta \in [0,1)^s$, where $\bsDelta = (\Delta_1, \ldots, \Delta_s)$, to all of the points of a lattice rule, i.e., for $k = 0,\ldots,N-1$,
\begin{align}\label{eq:shifted}
  \bsx_k
  &=
  \left(\frac{\bsz k}{N} + \bsDelta\right) \bmod{1}
  =
  ( (z_1 k / N + \Delta_1) \bmod{1}, \ldots, (z_s k / N + \Delta_s) \bmod{1} )
  ,
\end{align}
the rule is called a \emph{(randomly-)shifted lattice rule}.
For functions which can be expressed in terms of a Fourier series such a shift changes the error for a fixed function, as the sum
\begin{multline*}
  \frac1N \sum_{k=0}^{N-1} \varphi_{\bsh}(\bsx_k)
  =
  \frac1N \sum_{k=0}^{N-1} \E^{\twopii (\bsz k / N + \bsDelta) \cdot \bsh}
  \DG{ \\ }{}
  =
  \E^{\twopii \bsDelta \cdot \bsh} \frac1N \sum_{k=0}^{N-1} \E^{\twopii \bsz \cdot \bsh k / N}
  =
  \E^{\twopii \bsDelta \cdot \bsh} \, \ind{\bsh \in L^\perp}
  ,
\end{multline*}
where $\ind{}$ is the indicator function,
and so\RefEq{eq:error} becomes
\begin{align*}
  E_N(f; \bsz, \bsDelta)
  &=
  \sum_{\bszero \ne \bsh \in L^\perp} \E^{\twopii \bsDelta \cdot \bsh} \, \hat{f}_{\bsh}
  .
\end{align*}
The worst-case error however stays unchanged for the Fourier space as, for $1 < p \le \infty$,
\begin{align*}
  e(\bsz, N; |||\cdot|||_{p,\alpha,\bsgamma})
  &=
  \left(\sum_{\bszero \ne \bsh \in L^\perp}  r_{\alpha,\bsgamma}(\bsh)^{-q} \left| \E^{\twopii \bsDelta \cdot \bsh} \right|^q\right)^{1/q}
  \DG{ \\ }{}
  &  % don't know how to get rid of the alignment character for non DG
  \DG{ }{ \hspace{-5mm} }
  =
  \left(\sum_{\bszero \ne \bsh \in L^\perp}  r_{\alpha,\bsgamma}(\bsh)^{-q} \right)^{1/q}
  ,
\end{align*}
and similarly for $p=1$ and $q=\infty$.
By observing that the expected value of a randomly-shifted lattice rule vanishes, $\mathbb{E}_{\bsDelta}[\E^{\twopii \bsDelta \cdot \bsh}] = 0$, where the shift is uniformly distributed over the unit cube, i.e., $\bsDelta \sim U[0,1)^s$, it is possible to obtain a statistical error estimator by drawing $\nu$ i.i.d.\ random shifts and averaging the obtained approximations,
\begin{align*}
  \overline{Q}_{N,\nu}(f; \bsz)
  =
  Q_{N,\nu}(f; \bsz, \{\bsDelta_i\}_{i=1}^\nu)
  &=
  \frac1\nu \sum_{i=1}^\nu Q_N(f; \bsz, \bsDelta_i)
  ,
  \\\text{where }
  Q_N(f; \bsz, \bsDelta_i)
  &=
  \frac1N \sum_{k=0}^{N-1} f((\bsz k / N + \bsDelta_i) \bmod{1})
  ,
\end{align*}
the \emph{standard error} of these independent approximations is then given by
\begin{align*}
  \sqrt{\frac1{\nu(\nu-1)} \sum_{i=1}^\nu \left( Q_N(f; \bsz, \bsDelta_i) - \overline{Q}_{N,\nu}(f; \bsz) \right)^2}
  ,
\end{align*}
and can be used in, e.g., a Chebyshev confidence interval, see, e.g., \cite[p.~91]{SJ94}.

Randomly-shifted lattice rules have a purpose for non-periodic functions as well.
For $s=1$, $p=2$ and integer $r \ge 1$ consider the norm of the \emph{unanchored Sobolev space} of smoothness $r$
\begin{align*}
  \|f\|_{2,r,\gamma}^{2}
  &=
  \left|\int_0^1 f(x) \rd{x}\right|^2 
  + \gamma^{-1} \sum_{\tau=1}^{r-1} \left|\int_0^1 f^{(\tau)}(x) \rd{x}\right|^2 
  + \gamma^{-1} \int_0^1 \left| f^{(r)}(x) \right|^2 \rd{x}
  ,
\end{align*}
and its tensor generalization for $s \ge 2$.
Through the theory of reproducing kernel Hilbert spaces it can be shown that the shift-averaged kernel of this space is a sum of Bernoulli polynomials of even degrees, starting from degree~2 up to degree $2r$.
More specifically, for $r=1$ the reproducing kernel coincides with that of a Korobov space with $\alpha=1$ and the weights scaled by $1/(2 \pi^2)$, i.e., here, for $p=2$,
\begin{align*}
  \omega(x)
  &= 
  2\pi^2 \sum_{0 \ne h \in \Z} \frac{\E^{\twopii h x}}{|h|^{2}}
  =
  B_2(x)
  =
  x^2 - x + \tfrac16
  ,
\end{align*}
where $B_2(x)$ is the Bernoulli polynomial of degree~2, compare with\RefEq{eq:Korobov-omega}.
This means a randomly-shifted lattice rule is expected to achieve the optimal rate for $r=\alpha=1$ being $O(N^{-1+\epsilon})$, $\epsilon > 0$, see, e.g., \cite{NW2008}.
Furthermore all tractability results can be transferred from one space to the other.
For higher order unanchored (non-periodic) Sobolev spaces the random shifting does not help and so a randomly-shifted lattice rule is stuck with the rate for $r=1$.

\subsection{Tent-transformed lattice rules and the cosine space}\label{sec:tent}

As discussed in \RefSec{sec:lattice-rules} the choice of the standard Fourier basis reduces the sum\RefEq{eq:sum-basis-funs} to the dual lattice condition.
The effect of this choice of basis is that functions with an absolutely convergent Fourier series expansion are by definition periodic functions.
It is possible to pick a different basis and express the functions in a cosine expansion
\begin{multline}\label{eq:cos-series}
  f(\bsx)
  =
  \sum_{\bsh \in \N_0^s} \hat{f}_{\bsh} \, \prod_{j \in \setu(\bsh)} \kappa \cos(\pi h_j x_j)
  ,
  \DG{ \\ }{\quad}
  \text{ where } \quad
  \hat{f}_{\bsh}
  =
  \int_{[0,1)^s} f(\bsx) \, \prod_{j \in \setu(\bsh)} \kappa \cos(\pi h_j x_j) \rd\bsx
  ,
\end{multline}
with $\kappa \ne 0$ an arbitrary constant, e.g., $\kappa = \sqrt2$.
Functions in this space can be non-periodic, in fact, this cosine basis spans $L_2([0,1)^s)$.
Such a space was studied in \cite{DNP2013} in the Hilbert setting.
To regain the nice property of the dual lattice it is easy to show that the component wise application of the \emph{tent transform}
\begin{align*}
  \phi(x)
  &:=
  1 - |2x - 1|
\end{align*}
to the lattice rule point set, obtaining a ``tent-transformed lattice rule'', reduces the sum\RefEq{eq:sum-basis-funs} to the dual lattice condition as well.
This is shown in the next theorem which is a slight generalization of the result in \cite{DNP2013}.

\begin{theorem}[Tent-transformed lattice rule error bound]
Suppose $f$ can be expanded in an absolutely convergent cosine series\RefEq{eq:cos-series}, then using a tent-transformed lattice rule, the error of approximating the integral is given by
\begin{align*}
  E_N(f; \bsz, \phi)
  &=
  \frac1N \sum_{k=0}^{N-1} f(\phi(z_1 k / N \bmod1), \ldots, \phi(z_s k / N \bmod1))
  -
  \int_{[0,1)^s} f(\bsx) \rd\bsx
  \\
  &=
  \sum_{\substack{\bszero \ne \bsh \in \Z^s \\ \bsh \cdot \bsz \equiv 0 \tpmod{N}}} 
  \hat{f}_{|\bsh|} \, (\kappa/2)^{|\setu(\bsh)|}
  .
\end{align*}
Furthermore define, for $1 \le p < \infty$,
  \begin{align*}
    ||| f |||_{p,\alpha,\bsgamma}^p
    &=
    \sum_{\bszero \ne \bsh \in \N_0^s} |\hat{f}_{\bsh}|^p \, r_{\alpha,\bsgamma}(\bsh)^p
    ,
  \end{align*}
  and, for $p=\infty$,
  \begin{align*}
    ||| f |||_{\infty,\alpha,\bsgamma}
    &=
    \sup_{\bszero \ne \bsh \in \N_0^s} |\hat{f}_{\bsh}| \, r_{\alpha,\bsgamma}(\bsh)
    .
  \end{align*}
Then for $\frac1p + \frac1q = 1$,
\begin{align*}
  E_N(f; \bsz, \phi) 
  &\le
  |||f|||_{p,\alpha,\bsgamma}
  \,
  e(\bsz, N, \phi; |||\cdot|||_{p,\alpha,\bsgamma})
\end{align*}
where, for $1 < p \le \infty$,
\begin{align*}
  e(\bsz, N, \phi; |||\cdot|||_{p,\alpha,\bsgamma})
  &=
  \left( 
  \sum_{\substack{\bszero \ne \bsh \in \Z^s \\ \bsh \cdot \bsz \equiv 0 \tpmod{N}}} 
  r_{\alpha,\bsgamma}(\bsh)^{-q}
  \,
  (\kappa^q/2)^{|\setu(\bsh)|}
  \right)^{1/q}
  ,
\end{align*}
and for $p=1$ and $q=\infty$
\begin{align*}
  e(\bsz, N, \phi; |||\cdot|||_{1,\alpha,\bsgamma})
  &=
  \sup_{\substack{\bszero \ne \bsh \in \Z^s \\ \bsh \cdot \bsz \equiv 0 \tpmod{N}}} 
  r_{\alpha,\bsgamma}(\bsh)^{-1}
  \,
  \kappa^{|\setu(\bsh)|}
  .
\end{align*}
In fact, the worst-case error for tent-transformed lattice rules in the cos-space equals the worst-case errors for lattice rules in the Korobov space for the choice of $\kappa = 2^{1/q}$, $1 \le q \le \infty$.
\end{theorem}
\begin{proof}
Since, for any $h \in \N_0$,
\begin{align*}
  \cos(\pi h \phi(x))
  &=
  \cos(2\pi h x)
  ,
  \qquad
  \text{for all } 0 \le x \le 1
  ,
\end{align*}
it follows that, for any $\bsh \in \N_0^s$,
\begin{align*}
  \prod_{j \in \setu(\bsh)} \kappa \cos(\pi h_j \phi(x_j))
  &=
  \prod_{j \in \setu(\bsh)} \kappa \cos(2 \pi h_j x_j)
  \\
  &=
  (\kappa/2)^{|\setu(\bsh)|}
  \prod_{j \in \setu(\bsh)} \left( \E^{\twopii h_j x_j} + \E^{-\twopii h_j x_j} \right)
  \\
  &=
  (\kappa/2)^{|\setu(\bsh)|}
  \sum_{\bssigma_{\setu} \in \{\pm1\}^{|\setu(\bsh)|}} 
  \prod_{j \in \setu(\bsh)} \E^{\twopii \sigma_j h_j x_j}
  ,
\end{align*}
where for $\setu = \emptyset$, i.e., $\bsh = \bszero$, the sum over $\bssigma_{\setu}$ is to be interpreted as the identity operator.
Obviously, and with the same interpretation when $\setu = \emptyset$ and $A$ an arbitrary function,
\begin{align*}
  \sum_{\bsh \in \Z^s} A(\bsh)
  &=
  \sum_{\bsh \in \N_0^s} 
  \sum_{\bssigma_{\setu} \in \{\pm1\}^{|\setu(\bsh)|}} 
  A(\bssigma_\setu \bsh)
  ,
  &
  \text{where }
  (\bssigma_\setu \bsh)_j
  &=
  \begin{cases}
    \sigma_j h_j & \text{if } j \in \setu , \\
    h_j          & \text{otherwise} .
  \end{cases}
\end{align*}
Thus, since for $\bsh \in \N_0^s$ and any $\sigma_\setu \in \{\pm1\}^{|\setu(\bsh)|}$, $\hat{f}_{\bsh} = \hat{f}_{|\sigma_\setu \bsh|}$ and $|\setu(\sigma_\setu \bsh)| = |\setu(\bsh)|$, then it follows that
\begin{align*}
  E_N(f; \bsz, \phi)
  &=
  \sum_{\bszero \ne \bsh \in \N_0^s}
  \hat{f}_{\bsh}
  \,
  \frac1N \sum_{k=0}^{N-1} \prod_{j \in \setu(\bsh)} \kappa \cos(\pi h_j \phi(x_j^{(k)}))
  \\
  &=
  \sum_{\bszero \ne \bsh \in \N_0^s}
  \hat{f}_{\bsh}
  \,
  (\kappa/2)^{|\setu(\bsh)|}
  \sum_{\bssigma_\setu \in \{\pm1\}^{|\setu(\bsh)|}}
  \frac1N \sum_{k=0}^{N-1}
  \prod_{j \in \setu(\bsh)} \E^{\twopii \sigma_j h_j x_j^{(k)}}
  \\
  &=
  \sum_{\bszero \ne \bsh \in \N_0^s}
  \sum_{\bssigma_\setu \in \{\pm1\}^{|\setu(\bsh)|}}
  \hat{f}_{|\bssigma_\setu \bsh|}
  \,
  (\kappa/2)^{|\setu(\bssigma_\setu\bsh)|}
  \frac1N \sum_{k=0}^{N-1}
  \E^{\twopii (\bssigma_\setu \bsh) \cdot \bsx^{(k)}}
  \\
  &=
  \sum_{\bszero \ne \bsh \in \Z^s}
  \hat{f}_{|\bsh|}
  \,
  (\kappa/2)^{|\setu(\bsh)|}
  \frac1N \sum_{k=0}^{N-1}
  \E^{\twopii \bsh \cdot \bsz k / N}
  \\
  &=
  \sum_{\bszero \ne \bsh \in L^\perp}
  \hat{f}_{|\bsh|}
  \,
  (\kappa/2)^{|\setu(\bsh)|}
  ,
\end{align*}
where $L^\perp = \{ \bsh \in \Z^s: \bsh \cdot \bsz \equiv 0 \pmod{N} \}$ is the dual of the original lattice rule.
Applying H\"older's inequality to
\begin{align*}
  E_N(f; \bsz, \phi)
  &=
  \sum_{\bszero \ne \bsh \in \Z^s}
  \left(
  \hat{f}_{|\bsh|} \, 2^{-|\setu(\bsh)|/p} \, r_{\alpha,\bsgamma}(\bsh)
  \right)
  \left(
  r_{\alpha,\bsgamma}(\bsh)^{-1} \, \ind{\bsh \in L^\perp} \, 2^{-|\setu(\bsh)|/q} \, \kappa^{|\setu(\bsh)|}
  \right)
\end{align*}
yields the result, with equality to the Korobov worst-case error for the choice $\kappa = 2^{1/q}$.
\end{proof}

The tent transform (under the name \emph{baker's transform}) occurred in \cite{Hic2002} to attain respectively $O(N^{-1+\epsilon})$ and $O(N^{-2+\epsilon})$ convergence for randomly-shifted and then tent-transformed lattice rules in the unanchored Sobolev space of smoothness $r=1$ and $r=2$.
In \cite{DNP2013} it was shown however that no random shifting is needed for the case $r=1$ as the cosine space with $\alpha=1$ and $\kappa=\sqrt2$ coincides with the unanchored Sobolev space with $r=1$ with the weights scaled by $1/\pi^2$.

\subsection{Polynomial lattice rules and Walsh spaces}\label{sec:poly-spaces}

In \RefSec{sec:poly-lattice-rules} the \emph{Walsh space} was already mentioned.
The $r_{\alpha,\bsgamma}$ function takes the form, $\alpha > 1/q$,
\begin{align*}
  r_{\alpha,\bsgamma}(\bsh)
  &=
  \gamma_{\setu(\bsh)}^{-1/2} \prod_{j \in \setu(\bsh)} b^{\alpha \floor{\log_b h_j}}
  ,
\end{align*}
and so the function $\chi$ from\RefEq{eq:poly-chi} takes the form
\begin{align*}
  \chi(\bsx; |||\cdot|||_{p,\alpha,\bsgamma})
  &=
  \sum_{\emptyset \ne \setu \subseteq \{1:s\}}
  \gamma_{\setu}^{q/2}
  \prod_{j\in\setu}
  \omega(x_j)
  ,
  \qquad \text{ where }
  \omega(x)
  =
  \sum_{h=1}^\infty
  \frac{\wal_{b,h}(x)}{b^{q \alpha\floor{\log_b h}}}
  .
\end{align*}
The infinite sum above can be written in closed form, see \cite{DKPS2005}.
E.g., for $\alpha=2$, $p=\infty$ and $q=1$, and $b=2$ (the most practical case on a digital computer) this becomes
\begin{align}\label{eq:omega-walsh}
  \omega(x)
  =
  \sum_{h=1}^\infty
  \frac{\wal_{2,h}(x_j)}{2^{2 \floor{\log_2 h}}}
  &=
  12 \left( \tfrac16 - 2^{\floor{\log_2 x} - 1} \right)
  .
\end{align}
As is the case for lattice rules, also here a quantity $R_\alpha$ could be defined for which it is not needed that $\alpha > 1/q$ by using a finite dimensional basis.
The basis is then restricted to $\Lambda = \{0,\ldots,N-1\}^s$.
This is typically done for $\alpha = 1$ and a modified (and unweighted) $r_1(h) = \prod_{j=1}^s \rho_b(h_j)$ with $\rho_b(0) = 1$ and $\rho_b(h) = (b^{a+1} \sin(\pi h_a / b))^{-1}$ for $h = h_0 + \cdots + h_a b^a$ and $h_a \ne 0$, i.e., $a = \floor{\log_b h}$.
For $b=2$ this matches with $r_{\alpha,\bsgamma}$ from above with $\gamma_\setu \equiv 1$ and $\alpha = 1$.
See, e.g., \cite{Nie92,DP2010,Pil2012}, with slight variations depending on the source.

Random shifting can also be done for polynomial lattice rules.
A \emph{digitally-shifted polynomial lattice rule} adds a shift $\bsDelta \in [0,1)^s$ to all of the points, as in\RefEq{eq:shifted}, but in a digital way:
\begin{align*}
  \bsy'_k
  &=
  \bsy_k \oplus_b \bsDelta
  ,
  &
  y'_{k,j,i}
  &=
  (y_{k,j,i} + \Delta_{j,i}) \bmod{b}
  ,
\end{align*}
where in the last equation the $y_{k,j,i}$ are the base~$b$ digits of $y_{k,j}$ and similar for $\Delta_{j,i}$ and $y'_{k,j,i}$.
(Note that due to the finite expansion of the $y_{k,j}$, i.e., up to $n$ base~$b$ digits, this is well defined.)
Using the technology of reproducing kernel Hilbert spaces it was shown in \cite{DP2005} that the reproducing kernel of the unanchored Sobolev space with smoothness $r=1$ results in a worst-case integrand which matches the Walsh space. 
E.g., for $b=2$
\begin{align*}
  \omega(x)
  &=
  \tfrac16 - 2^{\floor{\log_2 x} - 1}
  ,
\end{align*}
compare with\RefEq{eq:omega-walsh}.
This agrees with the similar case of lattice rules in \RefSec{sec:randomly-shifted}.

The power of polynomial lattice rules rest in the fact that it is possible to also embed certain Sobolev spaces with $r\ge2$ (but not $r=1$) into a special Walsh space, see \cite{Dic2008}.
For this the $r_{\alpha,\bsgamma}$ function needs to take on a slightly more complicated form.
Consider the unique base~$b$ expansion of $h \in \N_0$ written as
\begin{align*}
  h = (\cdots h_2 h_1 h_0)_b
  =
  \sum_{i=0}^{\infty} h_i \, b^i
  &=
  \sum_{i=1}^{\#h} h_{a_i} \, b^{a_i}
  ,
\end{align*}
where $\#h$ is the number of non-zero base~$b$ digits in the unique expansion of $h$, so all $h_{a_i} \in \{1,\ldots,b-1\}$ and $a_1 > \cdots > a_{\#h} \ge 0$, $a_1 = \floor{\log_b h}$.
With this representation in mind and for fixed integer $\alpha \ge 1$ now define the one-dimensional function
\begin{align*}
  r_\alpha(h)
  &=
  b^{\sum_{i=1}^{\min(\#h, \alpha)} (a_i+1)}
\end{align*}
and the weighted product for the multivariate version $r_{\alpha,\bsgamma}(\bsh) = \gamma_{\setu(\bsh)}^{-1/2} \prod_{j \in \setu(\bsh)} r_\alpha(h_j)$.
This defines what is called a \emph{higher order Walsh space}.
Now if $n=\alpha m$ the worst-case error in this space can be bounded by $O(N^{-\alpha+\epsilon})$, $\epsilon > 0$, for digital nets and also for polynomial lattice rules, see \cite{Dic2008,DP2007}.
In \cite{Dic2008} it is shown that functions in Sobolev spaces with $r \ge 2$ have Walsh coefficients which decay faster than the above $r_{\alpha,\bsgamma}(\bsh)^{-1}$ and thus for $p=\infty$ these spaces are embedded in the higher order Walsh space.
In \cite{BDLNP2012} the following worst-case functions were obtained explicitly for $b=2$
  \begin{align*}
    \omega_2(x)
    &=
    s_1(x) + \tilde{s}_2(x)
    ,
    &
    \omega_3(x)
    &=
    s_1(x) + s_2(x) + \tilde{s}_3(x)
    ,
  \end{align*}
  where
  \begin{align*}
    s_1(x) &= 1 - 2 x
    ,
    \\
    s_2(x) &= 1/3 - 2 (1-x) x
    ,
    \\
    \tilde{s}_2(x) &= (1 - 5 t_1) / 2 - (a_1 - 2) x
    ,
    \\
    \tilde{s}_3(x) &= (1 - 43 t_2)/18 + (5 t_1 - 1) x + (a_1 - 2) x^2
    ,
  \end{align*}
  with, for $0 < x < 1$,
  \begin{align*}
    a_1 &= - \floor{\log_2(x)}
    ,
    &
    t_1 &= 2^{-a_1}
    ,
    &
    t_2 &= 2^{-2 a_1}
    ,
  \end{align*}
  and $a_1 = 0$, $t_1 = 0$ and $t_2 = 0$ when $x=0$.
Then $\omega_2(x)$ is the $\omega$ function for $\alpha = 2$ and $\omega_3(x)$ for $\alpha = 3$.
An algorithm to calculate $\omega(k/b^m)$, $k=0,\ldots,b^m-1$, for any $\alpha$ and $b$ is also given in \cite{BDLNP2012}.

Also the tent-transform can be applied to polynomial lattice rules in a similar form as in \cite{Hic2002}, see \RefSec{sec:tent} to improve the convergence rate, see \cite{CDLP2007}.

\section{Component-by-component constructions}\label{sec:construction}

Given the analysis from the previous sections it is now possible to try and find a generating vector $\bsz^*$ that achieves almost the best-possible convergence rate of the worst-case error.
For $1 < p \le \infty$ this is possible by minimizing the error of the function $\chi$ given by\RefEq{eq:chi}, or\RefEq{eq:poly-chi}, for all choices of $\bsz$.
However, the number of choices for $\bsz$ is excessively large, e.g., in the case of lattice rules, the number of choices is roughly $|\Z_N^s| = N^s$ and a similar statement is true for polynomial lattice rules.
Korobov \cite{Kor59} already found that the generating vector can be constructed component-by-component in the classical, i.e., unweighted, Korobov space.
This was later rediscovered and generalized by Sloan et al., e.g., see \cite{SR2002,SKJ2002,SKJ2001}.

\subsection{Component-by-component construction}

First some assertions are made which are true for all the spaces defined in this manuscript.
The decay function can be split into a product of the weight $\gamma_{\setu(\bsh)}$ and a part determining the convergence of the series expansion $r_\alpha(\bsh)$, i.e., $r_{\alpha,\bsgamma}(\bsy) = \gamma_{\setu(\bsh)}^{-1/2} \, r_\alpha(\bsh)$.
Furthermore the unweighted part $r_\alpha(\bsh)$ can be written as a product and is ``zero neutral'' (or embedded), i.e., $r_\alpha(h_1,h_2,0) = r_\alpha(h_1,h_2)$.
In other words
\begin{align}\label{A1}
  r_\alpha(\bsh)
  &=
  \prod_{j \in \setu(\bsh)} r_\alpha(h_j)
  ,
\end{align}
where the functions for the one-dimensional parts could in fact be chosen differently if needed.
For definiteness, it is assumed that the series expansions converge with an algebraic rate such that for $h \ne 0$ the following holds
\begin{align}\label{eq:algebraic}
  r_\alpha(|G| h)
  &\ge
  N^\alpha \, r_\alpha(h)
  ,
\end{align}
where $G=\Z_N$ and thus $|G| = N$ for lattice rules in a Fourier space, $G=G_{b,m}$ and thus $|G| = N$ for polynomial lattice rules in the Walsh space and $G=G_{b,n}$ with $n=\alpha m$, and thus $|G| = b^{\alpha m}$ for higher order polynomial lattice rules in the higher order Walsh space.
A further assumption is that
\begin{align}\label{A3}
  \sum_{0 \ne h \in \Lambda_{\{j\}}} r_\alpha(h_j)^{-1} 
  &< \infty
  \qquad \text{ for all } \alpha > 1
  ,
\end{align}
where $\Lambda_{\{j\}} = \{ \bsh \in \Lambda : h_j \ne 0 \text{ and } h_{j'} = 0 \text{ for all } j' \ne j \}$.
These conditions are true for all $r_\alpha(\bsh)$ functions considered in this manuscript.

Similar assumptions as for $r_{\alpha}$ are made for the basis functions, i.e.,
\begin{align}\label{A2}
  \varphi_{\bsh}(\bsx)
  =
  \prod_{j \in \setu(\bsh)} \varphi_{h_j}(x_j)
  ,
\end{align}
and thus $\Lambda = \prod_{j = 1}^s \Lambda_{\{j\}}$, where again, the one-dimensional functions could be different for the different dimensions if needed.

The component-by-component algorithm constructs the generating vector of a (polynomial) lattice rule dimension by dimension: in each step extending the dimension by one and finding the best $z_s$, denoted below by $z_s^*$, while keeping all previous components of the generating vector $z_j^*$, $j < s$, fixed.
The first step for the component-by-component construction is to write the worst-case error in a recursive form:
\begin{align}\notag
  % don't know how to do this for non DG
  \hspace{1cm}&\hspace{-1cm}e_s(\bsz, N; |||\cdot|||_{p,\alpha,\bsgamma})^q
  \\\notag
  &=
  \sum_{\emptyset \ne \setu \subseteq \{1:s\}}
  \gamma_\setu^{q/2}
  \sum_{\bsh_\setu \in L^\perp_\setu}
  r_\alpha(\bsh_\setu)^{-q}
  \\\notag
  &=
  \sum_{\emptyset \ne \setu \subseteq \{1:s-1\}}
  \gamma_\setu^{q/2}
  \sum_{\bsh_\setu \in L^\perp_\setu}
  r_\alpha(\bsh_\setu)^{-q}
  +
  \sum_{s \in \setu \subseteq \{1:s\}}
  \gamma_\setu^{q/2}
  \sum_{\bsh_\setu \in L^\perp_\setu}
  r_\alpha(\bsh_\setu)^{-q}
  \\\label{eq:wce-recursion}
  &=
  e_{s-1}((z_1,\ldots,z_{s-1}), N; |||\cdot|||_{p,\alpha,\bsgamma})^q
  +
  \theta_s(z_s)
  ,
\end{align}
where the subscripts $s$ and $s-1$ were added to make the number of dimensions explicit and the dependence of $\theta_s$ on $z_j$ for $j < s$ is suppressed as they are considered already fixed when determining $z_s^*$.
It is clear that for the optimal choice of $z_s$, while keeping all previous choices fixed, only $\theta_s(z_s)$ needs to be evaluated.
Note that this step implicitly assumes that $r_\alpha$ is ``zero neutral''.

The following theorem shows that the component-by-component algorithm can find good rules. 
The proof is written such that it applies to both lattice rules and polynomial lattice rules.
\begin{theorem}[Component-by-component construction for prime $N$ or irreducible $P$]\label{thm:cbc}
  Assume that\RefEq{A1}--\eqref{A2} holds.
  Let in the case of lattice rules $G = \Z_N$, $N$ be prime, $\Lambda_\setu = \Z_\setu$ and $\varphi_{\bsh}$ the Fourier basis and, in the case of polynomial lattice rules $G = G_{b,n}$, $P$ be irreducible over $\F_b$, $\Lambda_\setu = \N_\setu$  and $\varphi_{\bsh}$ the Walsh basis.

  Then, for fixed $\lambda$ with $1 \le \lambda < \alpha q$, a generating vector $\bsz^* \in G^s$ can be found, component-by-component, minimizing the worst-case error in each step for each choice $z_s^*$, given the best previous choices $z_j^*$, for $j < s$.
  The worst-case error then satisfies for each~$s$
\begin{align*}
  e_s(\bsz^*, N; |||\cdot|||_{p,\alpha,\bsgamma})
  &\le
  \left(
  \frac2N
  \sum_{\emptyset \ne \setu \subseteq \{1:s\}}
  \gamma_\setu^{q/2\lambda}
  \sum_{\bsh_\setu \in \Lambda_\setu}
  r_\alpha(\bsh_\setu)^{-q/\lambda}
  \right)^{\lambda/q}
  .
  \end{align*}
\end{theorem}
\begin{proof}
The proof uses that for $\lambda \ge 1$ and positive $a_k$ the following holds
$
  \left(\sum_{k} a_k\right)^{1/\lambda}
  \le
  \sum_{k} a_k^{1/\lambda}
$
and the inequality is reversed in case $\lambda \le 1$.
This is often called Jensen's inequality.
For fixed $\lambda \ge 1$ the optimal choice in dimension $s$, denoted $z_s^*$, should do at least as good as the average over all possible choices $z_s \in G$ and this still holds if all quantities are risen to the power $1/\lambda \le 1$. 
Under the stated assumptions for $s=1$:
  \begin{align} \DG{ \notag }{}
    e_1(z_1^*, N; |||\cdot|||_{p,\alpha,\bsgamma})
    &=
    \left(\gamma_{\{1\}}^{q/2} \sum_{0 \ne h \in \Lambda} r_\alpha(|G|h)^{-q}\right)^{1/q}
    \DG{ \\ }{} \label{eq:induction-hypot}
    & % don't know how to get rid of the align for non DG
    \DG{}{\hspace{-3mm}}
    \le
    \left(\frac2N \, \gamma_{\{1\}}^{q/2\lambda} \sum_{0 \ne h \in \Lambda} r_\alpha(h)^{-q/\lambda}\right)^{\lambda/q}
    .
  \end{align}
For $s\ge2$ the optimal choice $z_s^*$ satisfies 
%\DG{$\left(\theta_s(z_s^*)\right)^{1/\lambda}$}{}
\begin{align*}
%\DG{}{\left(\theta_s(z_s^*)\right)^{1/\lambda}}
  &
  \DG{}{\hspace{-5mm}}
  \left(\theta_s(z_s^*)\right)^{1/\lambda}
  \\
  &\le
  \frac1{|G|} \sum_{z_s \in G} \left(\theta_s(z_s)\right)^{1/\lambda}
  \\
  &\le
  \frac1{|G|} \sum_{z_s \in G}
  \sum_{s \in \setu \subseteq \{1:s\}}
  \gamma_\setu^{q/2\lambda}
  \sum_{\bsh_\setu \in L^\perp_\setu}
  r_\alpha(\bsh_\setu)^{-q/\lambda}
  \\
  &=
  \!\!
  \sum_{s \in \setu \subseteq \{1:s\}}
  \gamma_\setu^{q/2\lambda}
  \sum_{\bsh_\setu \in \Lambda_\setu}
  r_\alpha(\bsh_\setu)^{-q/\lambda}
  \frac1N \sum_{k=0}^{N-1}
  \prod_{s \ne j \in \setu} \varphi_{h_j}(x_{k,j}(z_j))
  \frac1{|G|} \sum_{z_s \in G}
  \varphi_{h_s}(x_{k,s}(z_s)) 
  \\
  &\le
  \!\!
  \sum_{s \in \setu \subseteq \{1:s\}}
  \!\!
  \gamma_\setu^{q/2\lambda}
  \sum_{\bsh_\setu \in \Lambda_\setu}
  \!\!
  r_\alpha(\bsh_\setu)^{-q/\lambda}
  \frac1N 
  \!\!
  \sum_{k=0}^{N-1}
  \left| \prod_{s \ne j \in \setu} 
  \!\!
  \varphi_{h_j}(x_{k,j}(z_j)) \right| 
  \left| 
  \frac1{|G|} 
  \!\!
  \sum_{z_s \in G}
  \!
  \varphi_{h_s}(x_{k,s}(z_s)) 
  \right|
  \\
  &\le
  \!\!
  \sum_{s \in \setu \subseteq \{1:s\}}
  \gamma_\setu^{q/2\lambda}
  \sum_{\bsh_\setu \in \Lambda_\setu}
  r_\alpha(\bsh_\setu)^{-q/\lambda}
  \frac1N \sum_{k=0}^{N-1}
  \begin{cases}
    1 & \text{if $k = 0$ or $h_s \equiv 0 \pmod{|G|}$}, \\
    0 & \text{otherwise}
  \end{cases}
  \\
  &=
  \!\!
  \sum_{s \in \setu \subseteq \{1:s\}}
  \gamma_\setu^{q/2\lambda}
  \left[
  \sum_{\substack{\bsh_\setu \in \Lambda_\setu \\ h_s \not\equiv 0 \tpmod{|G|}}}
  \hspace{-4mm}\frac{r_\alpha(\bsh_\setu)^{-q/\lambda}}{N}
  +
  \sum_{\bsh_\setu \in \Lambda_\setu}
  \prod_{s \ne j \in \setu} r_\alpha(h_j)^{-q/\lambda}
  \, r_\alpha(|G| h_s)^{-q/\lambda}
  \right]
  \\
  &\le
  \left(\frac1{N^{\alpha q / \lambda}} + \frac1N\right)
  \sum_{s \in \setu \subseteq \{1:s\}}
  \gamma_\setu^{q/2\lambda}
  \sum_{\bsh_\setu \in \Lambda_\setu}
  r_\alpha(\bsh_\setu)^{-q/\lambda}
  \\
  &\le
  \frac2N
  \sum_{s \in \setu \subseteq \{1:s\}}
  \gamma_\setu^{q/2\lambda}
  \sum_{\bsh_\setu \in \Lambda_\setu}
  r_\alpha(\bsh_\setu)^{-q/\lambda}
  ,
\end{align*}
where\RefEqTwo{eq:sum-over-ZN}{eq:sum-over-Gbn} were used as well as\RefEq{eq:algebraic} and $\alpha q / \lambda \ge 1$.
For convergence reasons, see\RefEq{A3} and the induction hypothesis\RefEq{eq:induction-hypot}, it is needed that $\alpha q / \lambda > 1$.
The result now holds by induction on\RefEq{eq:wce-recursion} as
\begin{align*}
  e_s(\bsz^*, N; |||\cdot|||_{p,\alpha,\bsgamma})^q
  &=
  e_{s-1}(\bsz^*, N; |||\cdot|||_{p,\alpha,\bsgamma})^q
  +
  \theta_s(z_s^*)
  \\
  &\le
  \left(
  \frac2N
  \sum_{\emptyset \ne \setu \subseteq \{1:s-1\}}
  \gamma_\setu^{q/2\lambda}
  \sum_{\bsh_\setu \in \Lambda_\setu}
  r_\alpha(\bsh_\setu)^{-q/\lambda}
  \right)^\lambda
  \\
  &\qquad\qquad+
  \left(
  \frac2N
  \sum_{s \in \setu \subseteq \{1:s\}}
  \gamma_\setu^{q/2\lambda}
  \sum_{\bsh_\setu \in \Lambda_\setu}
  r_\alpha(\bsh_\setu)^{-q/\lambda}
  \right)^\lambda
  \\
  &\le
  \left(
  \frac2N
  \sum_{\emptyset \ne \setu \subseteq \{1:s\}}
  \gamma_\setu^{q/2\lambda}
  \sum_{\bsh_\setu \in \Lambda_\setu}
  r_\alpha(\bsh_\setu)^{-q/\lambda}
  \right)^\lambda
  . \qedhere
\end{align*}
\end{proof}

For all the spaces considered here, this results in the optimal rate of $O(N^{-\alpha+\epsilon})$, $\epsilon > 0$, see, e.g., \cite{NW2008}.
More explicitly, \RefThm{thm:cbc} gives the following result: For any $1/q \le \lambda < \alpha$ the worst-case error fulfills
\begin{align*}
  e_s(\bsz^*, N; |||\cdot|||_{p,\alpha,\bsgamma})
  &\le
  C_{s,\alpha,\bsgamma,\lambda} \,
  \frac{2^{\lambda}}{N^{\lambda}}
  ,
\end{align*}
where $C_{s,\alpha,\bsgamma,\lambda}$ depends on the weights and on the dimension.
If the weights decay fast enough then $C_{s,\alpha,\bsgamma,\lambda}$ can be replaced by an absolute constant $C_{\alpha,\bsgamma,\lambda}$ independent of $s$ and then the curse is vanished, see, 
for lattice rules, \cite{KS2005,DKS2013,NW2008,SKJ2001,SKJ2002,SW98,Woz2009,Was2013} for the conditions on the weights.
All these results, except \cite{Woz2009} and \cite{Was2013}, consider $p=2$ only, while \cite{Woz2009} also analyzes $p=\infty$ with not so dramatic changes and \cite{Was2013} studies tractability for general $p$.
For polynomial lattice rules see, e.g., \cite{DP2007,DKS2013}.

\subsection{Fast component-by-component construction}\label{sec:fast-cbc}

The fast component-by-component construction is a particular method of calculating the worst-case errors in the component-by-component construction from the previous section which makes use of fast convolution (by means of FFTs).
The fast component-by-component construction was first established for lattice rules with $N$ prime in \cite{NC2006-prime} and later extended for non-prime $N$ in \cite{NC2006}, see also \cite{Nuy2005}.
In \cite{NC2006-kernels} it was first demonstrated for polynomial lattice rules, see also \cite{Nuy2007}, and later for higher-order polynomial lattice rules in \cite{BDLNP2012}.
Also lattice sequences are possible, see \cite{CKN2006}.
The most recent construction is for SPOD weights \cite{DKGNS2013}.
Here only fixed rules with prime~$N$ or irreducible~$P(\X)$ will be considered.

Set
\begin{align*}
  \omega(x_{k,j})
  =
  \omega(k \cdot z_j)
  &=
  \sum_{0 \ne h \in \Lambda_{\{j\}}}
  r_\alpha(h)^{-q} \, \varphi_h(x_{k,j})
  ,
\end{align*}
where the notation $\omega(k \cdot z_j)$ is used to stress that this is a multiplication in the ring modulo~$N$ for lattice rules\RefEq{eq:lattice} or modulo~$P(\X)$ for polynomial lattice rules\RefEq{eq:poly-lattice}.
In particular, when $N$ is prime, or $P(\X)$ is irreducible, this is a field where $G \setminus \{0\} = \Z_N \setminus \{0\}$ or $\F_b[\X]/P(\X) \setminus \{ 0 \}$ is a cyclic multiplicative group.
The function $\omega$ was given in closed form for many spaces in \RefSec{sec:spaces}, but in principle it is sufficient if it can be calculated for each of $k/N$, $k=0,\ldots,N-1$.

Now from\RefEq{eq:wce-recursion} write
\begin{multline*}
  e_s((z_1^*, \ldots, z_{s-1}^*, z_s), N; |||\cdot|||_{p,\alpha,\bsgamma})^q
  \\=
  e_{s-1}((z_1^*, \ldots, z_{s-1}^*), N; |||\cdot|||_{p,\alpha,\bsgamma})^q
  +
  \theta_s(z_s)
\end{multline*}
where
\begin{align}\notag
  \theta_s(z_s)
  &=
  \sum_{\setu \subseteq \{1:s-1\}}
  \gamma_{\setu \cup \{s\}}
  \frac1N
  \sum_{k=0}^{N-1}
  Y_\setu(k)
  \,
  \omega(k \cdot z_s)
  =
  \frac1N
  \sum_{k=0}^{N-1}
  Y_s(k)
  \,
  \omega(k \cdot z_s)
  \\\label{eq:convolution}
  &=
  \frac1N \left( Y_s(0) \, \omega(0) + \sum_{k=1}^{N-1} Y_s(k) \, \omega(k \cdot z_s) \right)
\end{align}
with $Y_{\emptyset}(k) \equiv 1$ and 
\begin{align}\notag
  Y_\setu(k)
  &=
  \sum_{h_\setu \in \Lambda_\setu} \prod_{j\in\setu} r_\alpha(h_j)^{-q} \, \varphi_{h_j}(x_{k,j}(z^*_j))
  =
  \prod_{j \in \setu} \omega(k \cdot z_j^*)
  ,
  \intertext{and}
  \label{eq:Ys}
  Y_s(k)
  &=
  \sum_{\setu \subseteq \{1:s-1\}}
  \gamma_{\setu \cup \{s\}}
  \,
  Y_\setu(k)
  .
\end{align}
The sum in\RefEq{eq:convolution} for each choice of $z_s \in G \setminus \{ 0 \}$ is a circular convolution when expressing the element in terms of the generator $g$ of the cyclic group, $\langle g \rangle = G \setminus \{ 0 \}$ (as $N$ is prime or $P(\X)$ irreducible),
\begin{align*}
  \sum_{k=1}^{N-1} Y_s(k) \, \omega(k \cdot z_s)
  &=
  \sum_{\delta=0}^{N-2} Y_s(g^\delta \bmod{G}) \, \omega(g^{\delta + \vartheta} \bmod{G})
  ,
  \DG{}{\qquad} \text{ for all } z_s = g^{\vartheta} \in G \setminus \{0\}
  ,
\end{align*}
see also \cite{Nuy2005} for a very comprehensive explanation.
If $|G| > N$, as is the case for higher-order polynomial lattice rules, then this is in fact a sparse convolution, see \cite{BDLNP2012} for an analysis of this case.
Using FFTs the circular convolution can be done in $O(M \log M)$, with $M = |G|-1$, hence the annotation ``fast'' component-by-component construction.
For calculating $Y_s(k)$ the structure of the weights will be used.
The same weights from \RefSec{sec:weights} are here analysed with the cost per iteration of the component-by-component algorithm.
\begin{itemize}\setlength{\itemsep}{0mm}
  \item \emph{General weights}: for general weights there is no structure and so all values of $Y_\setu(k)$ need to be stored, which costs $O(2^s N)$ storage, and calculating\RefEq{eq:Ys} would cost $O(2^s N)$.
  \item \emph{Product weights}: for product weights $\gamma_\setu = \prod_{j\in\setu} \gamma_j$ it follows from\RefEq{eq:Ys} that
  \begin{align*}
    Y_s(k)
    &\hphantom{:}=
    \gamma_s \prod_{j=1}^{s-1} \bigl( 1 + \gamma_j \, \omega(x_{k,j}) \bigr)
    =
    \gamma_s \, P_{s-1}(k)
    ,
    \\
    \text{and }
    P_s(k)
    &:=
    \prod_{j=1}^s \bigl( 1 + \gamma_j \, \omega(x_{k,j}) \bigr)
    =
    \bigl(1 + \gamma_s \, \omega(x_{k,s}) \bigr) \, P_{s-1}(k)
    .
  \end{align*}
  The product can be stored and updated, i.e., overwritten, in each step of the component-by-component algorithm, needing a storage of $O(N)$ and an incremental calculation cost of $O(N)$.
  \item \emph{Order-dependent weights}: for order-dependent weights $\gamma_{\setu} = \Gamma_{|\setu|}$ it follows that
  \begin{align*}
    Y_s(k)
    &\hphantom{:}=
    \sum_{\ell=0}^{s-1} \Gamma_{\ell+1} 
    \left( \sum_{\substack{\setu \subseteq \{1:s-1\} \\ |\setu| = \ell}} \prod_{j\in\setu} \omega(x_{k,j}) \right)
    =
    \sum_{\ell=0}^{s-1} \Gamma_{\ell+1} 
    \,
    P_{s-1,\ell}(k)
    ,
    \\
    \text{and }
    P_{s,\ell}(k)
    &:=
    \sum_{\substack{\setu \subseteq \{1:s\} \\ |\setu| = \ell}} \prod_{j\in\setu} \omega(x_{k,j})
    =
    P_{s-1,\ell}(k)
    +
    P_{s-1,\ell-1}(k)
    \, 
    \omega(x_{k,s})
    .
  \end{align*}
  The $s$ sums of order $\ell=0,\ldots,s$ can be stored and updated, i.e., overwritten, in each step needing a storage $O(sN)$ and an incremental calculation cost of $O(sN)$.
  \item \emph{Finite-order-dependent weights}: for finite-order-dependent weights  $\Gamma_\ell = 0$ for $\ell > q^*$.
  The same analysis as above holds with $\ell=0,\ldots,q^*$.
  Storage cost is $O(q^*N)$ and the incremental calculation cost is also $O(q^*N)$.
  \item \emph{Product-and-order-dependent weights} (POD weights): for POD weights  $\gamma_{\setu} = \Gamma_{|\setu|} \prod_{j \in \setu} \beta_j$ the same result as for order-dependent weights is obtained as it does not matter if the function $\omega$ stays the same for different dimensions, it can be multiplied with $\beta_j$:
  \begin{align*}
    Y_s(k)
    &\hphantom{:}=
    \sum_{\ell=0}^{s-1} \Gamma_{\ell+1} 
    \left( \sum_{\substack{\setu \subseteq \{1:s-1\} \\ |\setu| = \ell}} \prod_{j\in\setu} \beta_j \, \omega(x_{k,j}) \right)
    ,
    \\
    \text{and }
    P_{s,\ell}(k)
    &:=
    \sum_{\substack{\setu \subseteq \{1:s\} \\ |\setu| = \ell}} \prod_{j\in\setu} \beta_j \,\omega(x_{k,j})
    =
    P_{s-1,\ell}(k)
    +
    P_{s-1,\ell-1}(k)
    \, \beta_j \,
    \omega(x_{k,s})
    .
  \end{align*}
  The costs are the same as for order-dependent weights: $O(sN)$ memory cost and $O(sN)$ update cost.
  \item \emph{Smoothness-driven product-and-order-dependent weights} (SPOD weights): the calculation for the SPOD weights is more involved, the reader is referred to \cite{DKGNS2013} where it is shown that, for the specific construction in that paper, the memory cost is $O(\alpha s N)$ and the update cost is $O(\alpha^{2} s N)$.
  \item \emph{Finite-diameter weights}: for a diameter $q$ the number of sets $\setu$ which include dimension $s$ and have non-zero weight is bounded by $2^{q-1}$. Moreover the smallest index in these sets is $s-q+1$, thus
  \begin{gather*}
    Y_s(k)
    =
    \sum_{\substack{s \in \setu \subseteq \{1:s\} \\ \operatorname{diam}(\setu) \le q}} \gamma_{\setu} \, Y_{\setu \setminus \{s\}}(k)
    =
    \sum_{\setu \subseteq \{\max(1,s-q+1):s-1\}} \gamma_{\setu \cup \{s\}}  \, Y_{\setu}(k)
    ,
    \\
    \text{and }
    Y_{\setu \cup \{s\} \setminus \{s-q+1\}}(k)
    =
    \begin{cases}
    Y_{\setu}(k) \, \omega(x_{k,s}) & \text{when } s < q, \\
    Y_{\setu}(k) \, \omega(x_{k,s}) / \omega(x_{k,s-q+1}) & \text{otherwise} .\\
    \end{cases}
  \end{gather*}
  The number of vectors $Y_{\setu}(k)$ is $2^{q-1}$ leading to a memory cost of $O(2^{q-1} N)$.
  They can be updated by exchanging the new index $s$ with the oldest index $s-q+1$ at a cost of $O(2^{q-2} N)$, but calculating $Y_s(k)$ will also cost $O(2^{q-1} N)$.
\end{itemize}
The \emph{finite-intersection weights} were left out as this constraint on its own is not sufficient to restrict the number of $Y_{\setu}(k)$ vectors to keep.
The same remark holds for plain \emph{finite-order weights}.
An obvious modification is, e.g., \emph{finite-POD weights} of order $q^{*}$ which would result in $O(q^{*}N)$ memory and $O(q^{*}N)$ update cost.

\begin{corollary}
  Fast component-by-component construction for a lattice rule or polynomial lattice rule with $N$ points in $s$ dimensions can be done in time $O(s \, |G| \log |G| +  s \, T N)$ and memory $O(T)$ where $|G| = N$ for a lattice rule or a polynomial lattice rule, and $N^\alpha$ for a higher-order polynomial lattice rule, and $T = 2^s$ for general weights, $T=1$ for product weights, $T=s$ for order-dependent weights and POD weights, $T=q^*$ for finite-order-dependent weights of order~$q^*$ and $T = 2^{q-1}$ for finite-diameter weights of diameter $q$.
\end{corollary}

The cost and memory constraints of the fast component-by-component algorithm are very reasonable except for higher-order polynomial lattice rules where $|G|$ scales exponentially with~$\alpha$.
A new construction \cite{GD2013} alleviates this problem by constructing interlaced higher-order digital nets based on polynomial lattice rules, see also \cite{DKGNS2013}.

\section{Conclusion}

Only lattice rules and polynomial lattice rules were considered in this manuscript, but a similar analysis can also be done for digital nets.
However, for digital nets the number of choices even in one dimension is already quite high and therefore a component-by-component construction does not make much sense unless extra structure is imposed on the generating matrices.
Like polynomial lattice rules impose a certain structure, so do shift nets \cite{Sch98}, cyclic nets \cite{Nie2004}, hyperplane nets \cite{PDP2006} and Vandermonde nets \cite{HN2013}.

Matlab implementations of the fast component-by-component algorithm can be found in \cite{NC2006-kernels} and \cite{Nuy2007} and the code is available from the author's website\footnote{\url{http://people.cs.kuleuven.be/~dirk.nuyens/fast-cbc/}}.
Also Matlab implementations for using such point sets are available from the author's website\footnote{\url{http://people.cs.kuleuven.be/~dirk.nuyens/qmc-generators/}}.

\DG{}
{
\section*{Acknowledgement}

\myack
}

\bibliographystyle{abbrv} 
\bibliography{mrabbrev,bib-abbrev,extra}

\end{document}